\documentclass[journal]{IEEEtran}
\usepackage{amssymb}

\usepackage[pdftex]{graphicx}
\usepackage{tikz}
\usepackage{amsmath}
\usepackage{framed}
\usepackage{bm}
\usepackage{cite}
\usepackage{comment}
\usepackage{algorithmic}
\usepackage{algorithm}
\usepackage{url}
\usepackage{tabularx}
\usepackage{booktabs}
\usepackage{array}

\usepackage{mathtools}
\mathtoolsset{showonlyrefs=true} %参照してない式番号を消す

\usepackage{amsfonts}
\usepackage{color}
\usepackage{cases}
\usepackage{enumerate}

\newtheorem{theorem}{Theorem}
\newtheorem{assumption}{Assumption}

\newtheorem{remark}{Remark}
\newtheorem{lemma}{Lemma}
\newtheorem{example}{Example}
\newtheorem{corollary}{Corollary}

\newtheorem{proposition}{Proposition}

\newtheorem{proof}{Proof}

\newcommand{\wh}[1]{\widehat{#1}}

\newcommand{\tblue}{\textcolor{blue}}
\usepackage{blkarray}

\newcommand{\argmin}{\mathop{\rm argmin}\limits}

\newcommand{\paren}[1]{\left(#1\right)}

\DeclareMathOperator{\diag}{diag}
\DeclareMathOperator{\tr}{tr}

\DeclareMathOperator{\Ker}{Ker}

\newcommand{\e}{\mathrm{e}}

\newcommand{\D}{\mathrm{d}}

\newcommand{\Pset}{\mathcal P}

\newcommand{\Real}{\mathbb{R}}

\newcommand{\bb}{\mathbb}
\newcommand{\R}{\mathbb{R}}

\usepackage{blkarray}

\usepackage{latexsym}
\def\qed{\hfill $\Box$}

\usepackage{cite}
\usepackage{amsmath}
\usepackage{url}
\begin{document}
\title{Controllability Allocation Scores\\ for Targeted Network Intervention}
\author{Kazuhiro Sato\thanks{K. Sato is with the Department of Mathematical Informatics, Graduate School of Information Science and Technology, The University of Tokyo, Tokyo 113-8656, Japan, email: kazuhiro@mist.i.u-tokyo.ac.jp}}
\maketitle
\thispagestyle{empty}
\pagestyle{empty}

\begin{abstract}
We introduce the controllability allocation score (CAS), a framework
for determining how intervention intensity should be distributed among
prescribed candidate input directions with respect to designated target
variables, together with the target controllability score (TCS) as its
nodewise specialization.
We establish existence of the CASs and develop a general uniqueness
theory based on restricted injectivity of the allocation-to-Gramian map,
including generic uniqueness with respect to the time horizon.
We show that
restricting attention to target variables can fundamentally alter the
optimal intervention allocation compared with the standard full-state
setting.
To enable scalability, we develop a general surrogate-Gramian
framework and derive objective-performance guarantees from relative
Gramian errors without requiring uniqueness or closeness of the CASs.
For the TCS specialization, we further construct a target-only
reduced virtual system and derive explicit bounds showing how the
approximation error depends on the coupling between target and non-target
nodes and on the dynamical growth rate.
Experiments on human brain networks show that the reduced formulation
accurately approximates the TCS at short horizons,
whereas the two controllability criteria exhibit markedly different
approximation accuracy at long horizons.
\end{abstract}
\begin{IEEEkeywords}
Controllability, convex optimization, network centrality, reduced-order modeling
\end{IEEEkeywords}

\IEEEpeerreviewmaketitle

\section{Introduction} \label{sec:intro}
Identifying the roles of individual nodes in complex systems---such as brain networks, social systems, or infrastructure systems---remains a central challenge in network science \cite{cimpeanu2023social, fang2016resilience, faramondi2020multi, kong2022influence, sporns2010networks, van2013network}. 
To address this, a wide range of centrality measures have been proposed. While these measures differ in scope and methodology, they can be broadly grouped into three categories: structural, spectral, and dynamics-aware. Structural and spectral measures---including degree, betweenness, closeness, eigenvector centrality, and PageRank---rely on network topology and, in some cases, edge weights, but ultimately provide static evaluations determined by the underlying structure \cite{bloch2023centrality, rodrigues2019network}.
By contrast, dynamics-aware measures go beyond topology by relying on the system matrix to capture time evolution. This dynamical perspective is crucial in systems where the propagation of signals, information, or energy evolves over time, meaning that node importance cannot be fully understood from static connectivity alone \cite{liu2011controllability, lu2016vital,tang2025network}.

Among such dynamics-aware approaches, controllability plays a central role, as it captures the fundamental ability of a system to be steered through external interventions \cite{kalman1960general, kalman1963mathematical}. 
Structural controllability provides a qualitative notion of controllability, focusing on whether a system is controllable for almost all numerical choices of edge weights based on the sparsity pattern of the system matrix together with the actuator placement \cite{lin1974structural, liu2011controllability}. Graph-theoretic tools such as maximum matchings and Dulmage--Mendelsohn decomposition are employed to determine the minimum number of inputs required for a network system to be controllable and to identify which state nodes should be actuated \cite{commault2024dilation,olshevsky2014minimal,pequito2015framework, ramos2022overview, terasaki2021minimal, terasaki2023minimal}.
In practice, however, structural controllability provides only a coarse guarantee. Even if a network is controllable in the qualitative sense, achieving control may require prohibitively large input energy, rendering the system effectively uncontrollable from a practical standpoint \cite{baggio2023controllability, pasqualetti2014controllability, summers2015submodularity}. This limitation highlights the need to move beyond a binary feasibility perspective and to quantify controllability in terms of the control effort required to steer the system.

In contrast, quantitative controllability enables the evaluation of the size of the reachable set and the amount of control energy required to reach a desired state, once the actual edge weights and time horizons are specified \cite{baggio2022energy, bof2016role, pasqualetti2014controllability, sato2020controllability, summers2015submodularity, lindmark2021centrality}.
However, existing dynamics-aware indices such as the Volumetric Control Energy (VCE) and Average Control Energy (ACE) centralities \cite{summers2015submodularity} have inherent limitations. They rely on single-input controllability Gramians, which often become nearly singular in large-scale networks \cite{baggio2023controllability, olshevsky2016eigenvalue, pasqualetti2014controllability, pasqualetti2019re, suweis2019brain,tu2018warnings}. 
As a result, these measures may fail to consistently reflect node importance---a shortcoming noted in \cite[Remarks~2 and~3]{sato2022controllability} and \cite[Section~IV-B]{sato2025uniqueness}.

To address this limitation, Sato and Terasaki \cite{sato2022controllability} introduced the controllability score (CS), a novel centrality metric defined for linear dynamical networks of the form
\begin{align}
\dot{x}(t) = Ax(t), \label{system0}
\end{align}
which quantifies the relative importance of each state node for full-state controllability. 
Here, $x(t)\in\mathbb{R}^n$ collects the state variables assigned to the $n$ nodes of the network, and the system matrix $A=(a_{ij})\in\mathbb{R}^{n\times n}$ describes the weighted network topology, where each entry $a_{ij}$ specifies how the state of node $j$ influences that of node $i$.
The CS was formulated under an idealized setting in which virtual inputs could, in principle, be applied to all state nodes of system \eqref{system0}. To formalize this, the virtual system
\begin{align}
\dot x(t) = Ax(t) + {\rm diag}(\sqrt{p_1},\ldots, \sqrt{p_n})u(t) \label{eq:lti}
\end{align}
was introduced, establishing a one-to-one correspondence between virtual input channels $u_i$ and state nodes $x_i$, thereby assigning to each node $x_i$ a weight $p_i$ that reflects the degree to which it can be directly actuated.  
This initial formulation was further developed in later work \cite{sato2025uniqueness}, which broadened the theoretical foundation of the CS and, through applications to brain networks, demonstrated that it uncovers node importance behaviors markedly different from those identified by conventional centrality measures.

While the CS provides a dynamics-aware measure of node importance for full-state controllability, practical intervention problems often require more than scoring individual nodes.
In particular, intervention may be restricted to selected nodes or
prescribed input directions
\cite{klickstein2017energy,gao2014target,guo2018novel,
manjunatha2024controlling}, while controllability may be relevant only
for selected or aggregated target variables.
Thus, a natural question is how intervention intensity should be
distributed among the available input directions with respect to the
variables of interest.
The candidate intervention directions and target variables need not
correspond to the same nodes or even have the same dimension.
Motivated by this setting, we introduce the controllability allocation
score (CAS) as a general framework for such intervention allocation.
We then introduce the target controllability score (TCS) as its nodewise
specialization, in which the candidate input directions and target
variables correspond to the same designated subset of network nodes,
thereby retaining the node-importance interpretation of the original CS.

This formulation raises two fundamental questions.
First, because the proposed scores are defined as optimal intervention
allocations, existence alone does not guarantee that the resulting
allocation is uniquely determined: distinct allocations may attain the same optimal
value.
Understanding when uniqueness holds is particularly important for the
TCS, whose components are interpreted as nodewise importance scores,
while for a general CAS it characterizes when the optimal intervention
allocation is uniquely determined.
Second, for large-scale networks, computing CASs from the full-order
dynamics can be computationally demanding, motivating efficient
surrogate formulations.
To make such approximations theoretically reliable, it is important to
quantify how surrogate errors affect performance under the original
controllability-allocation objectives.

The main contributions of this paper are as follows:
\begin{enumerate}
\item
We introduce the controllability allocation score (CAS), defined
as a solution to a finite-horizon convex optimization problem based on
the output controllability Gramian for allocating intervention intensity
among prescribed candidate input directions with respect to designated
target variables.
We establish existence of optimal solutions and develop a general
uniqueness theory based on restricted injectivity of the
allocation-to-Gramian map and an equivalent matrix rank condition.
We further show that satisfaction of this rank condition at one horizon
guarantees uniqueness for almost all positive horizons.

\item
We introduce the target controllability score (TCS) as a nodewise
specialization of the CAS framework, in which the candidate input
directions and target variables correspond to the same designated set
of network nodes.
For arbitrary system matrices \(A\), we prove that the TCS is uniquely
determined at all sufficiently small horizons and for almost all positive
horizons.
We further show that restricting controllability evaluation to target
nodes can fundamentally alter the optimal allocation, so that qualitative
properties of the standard full-state controllability scores studied in \cite[Section~III-B]{sato2025uniqueness} do not
generally carry over to the target setting.

\item
We develop a general surrogate framework for approximating
CASs through surrogate candidate output controllability Gramians and
derive quantitative guarantees relating relative Gramian errors to
the performance of surrogate allocations under the original CAS
objectives, without requiring uniqueness or closeness of the CASs.
For the TCS specialization, we construct a target-only reduced virtual
system and derive explicit error bounds that quantify how the
approximation accuracy depends on the coupling between target and
non-target nodes and on the dynamical growth rate.
We validate this approximation framework on human brain networks from
88 subjects.
The experiments show that the reduced formulation accurately
approximates the target-based scores at short horizons, whereas the
two controllability criteria exhibit markedly different approximation
accuracy at longer horizons.
\end{enumerate}

The remainder of this paper is organized as follows.
Section~\ref{Sec_TCS} 
introduces the CAS framework and its TCS specialization.
Section~III presents the computation and uniqueness analysis of the CASs and discusses qualitative differences between the target and standard VCS/AECS.
Section~\ref{sec:efficient} develops a general surrogate-Gramian framework and a target-only approximation for the TCS, together with quantitative error guarantees.
Section~\ref{Sec_numerical} validates the proposed framework using human brain network data from 88 individuals and demonstrates a trade-off between short-horizon approximation accuracy and long-horizon robustness.
Finally, Section~\ref{Sec_conclusion} concludes this paper.

{\it Notation:}
The sets of real numbers and positive real numbers are denoted by
${\bb R}$ and ${\bb R}_{>0}$, respectively.
For a matrix $X\in {\bb R}^{m\times n}$, $X^{\top}$ denotes the
transpose of $X$.
The symbols $O$ and $I$ denote the zero and identity matrices
of appropriate sizes, respectively.
For a square matrix $A\in {\bb R}^{n\times n}$, $\det A$ and
${\rm tr}(A)$ denote the determinant and trace of $A$,
respectively.
Given a vector $v=(v_i)\in {\bb R}^n$, $\|v\|$ and
${\rm diag}(v_1,\ldots, v_n)$ denote the usual Euclidean norm
$\|v\|=\sqrt{v^{\top}v}$ and the diagonal matrix with the diagonal
elements $v_1,\ldots, v_n$, respectively.
Instead of ${\rm diag}(v_1,\ldots, v_n)$, we also use
${\rm diag}(v)$.
For a matrix $A\in{\bb R}^{m\times n}$, we define its operator norm
(or spectral norm) induced by the Euclidean norm as
$\|A\| := \sup_{x\in{\bb R}^n\setminus\{0\}}
\frac{\|Ax\|}{\|x\|}$.
Equivalently, $\|A\|$ is the largest singular value of $A$.
For a symmetric matrix $A$,
$A\succeq O$ (resp. $A\succ O$) denotes that $A$ is positive
semidefinite (resp. positive definite), and $A\preceq O$
(resp. $A\prec O$) denotes that $A$ is negative semidefinite
(resp. negative definite).
More generally, for two symmetric matrices $A$ and $B$,
we write $A\preceq B$ (resp. $A\succeq B$) if $B-A$
(resp. $A-B$) is positive semidefinite.
This is the standard Loewner partial order.
The symbols $\lambda_{\rm max}(X)$ and $\lambda_{\rm min}(X)$
denote the largest and smallest eigenvalues of a symmetric matrix $X$,
respectively.
For a positive integer $r$, $\mathbb{S}^r$,
$\mathbb{S}_+^r$, and $\mathbb{S}_{++}^r$ denote the sets of real
symmetric, positive-semidefinite, and positive-definite
$r\times r$ matrices, respectively.
For \(X\in\mathbb{S}^r\), \({\rm svec}(X)\in
\mathbb{R}^{r(r+1)/2}\) denotes the symmetric vectorization of \(X\),
in which the off-diagonal entries are multiplied by \(\sqrt{2}\).
This convention satisfies
$\|{\rm svec}(X)\|=\|X\|_{\rm F}$, where
$\|\cdot\|_{\rm F}$ denotes the Frobenius norm.
For a matrix $X$, $\sigma_{\min}(X)$ denotes its smallest
singular value.
The vector $e_i^{(n)}$ denotes the $i$th standard basis vector
of $\mathbb{R}^n$, and $\mathbf{1}$ denotes the all-ones vector of
appropriate dimension.
For a set $\mathcal S$, $\operatorname{lin}\mathcal S$
denotes its linear span, and $\operatorname{ri}\mathcal S$
denotes its relative interior.

%%%%%%%%%%%
\section{Controllability Allocation Scoring} \label{Sec_TCS}

As discussed in Section~\ref{sec:intro}, practical network intervention is generally restricted to a prescribed set of input directions, while controllability is evaluated with respect to selected output variables,
rather than all nodes of system \eqref{system0}. 
Such restrictions commonly arise when intervention channels are localized or structurally constrained
and when controllability is of interest only for selected target variables.
In such cases, virtual system \eqref{eq:lti}, which assumes inputs at all state nodes and evaluates controllability with respect to the full state, may not be appropriate for intervention-oriented decision making.

Accordingly, we consider the following generalized intervention formulation:
\begin{align}
\begin{cases}
  \dot x(t) = Ax(t) + B(p)u(t)\\
  y(t) =Cx(t)
  \end{cases}\label{eq:lti_virtual3}
\end{align}
where
\begin{align}
B(p) &:=
B_0\diag(\sqrt{p_1},\ldots,\sqrt{p_m}), \label{eq:Bp}\\
B_0 &:=[\,b_1\ \cdots\ b_m\,]\in\R^{n\times m},
\end{align}
and $C\in\R^{r\times n}$ is a full-row-rank matrix specifying the target variables. The columns of $B_0$ specify the candidate intervention directions, whereas $C$ maps the state to the $r$-dimensional target output. Thus, the $m$ candidate input channels and the $r$ target variables are distinct objects in general, and their numbers need not coincide. For nodewise candidate inputs, the candidate and target node sets may coincide, overlap, or be disjoint, and no prescribed ordering between their cardinalities is required.
The allocation vector \(p\) assigns a nonnegative intervention weight
\(p_i\) to each candidate input direction \(b_i\in\R^n\) and is restricted to
\begin{equation}\label{eq:P}
p\in\Pset\subseteq\Delta_m
:=
\{p\in\R^m:p\geq0,\ \mathbf{1}^\top p=1\},
\end{equation}
where \(\Pset\) is nonempty, compact, and convex. 
In particular, \(p_i=0\) deactivates the
\(i\)th candidate channel without altering \(C\).

%%%%%%%%%%%%%%%%%%%%%%%%%
\subsection{Output Controllability} \label{Sec_output_con}

Before introducing the output controllability Gramian used to define the concept of CAS, we recall the notion of output controllability introduced in \cite{kreindler1964concepts}. 
System \eqref{eq:lti_virtual3} is said to be output controllable on $[0,T]$ if, for any initial state $x(0)$ and any desired terminal output $y_T$, there exists an input $u(\cdot)$ that steers the system so that $y(T)=y_T$ within the finite horizon $T$. 
In contrast, the usual notion of controllability requires that the entire state vector $x(t)$ can be driven to an arbitrary terminal state. 
Output controllability is therefore a weaker property, concerning only those components of the state observed through the output matrix $C$. 

As shown in \cite[Theorem I]{kreindler1964concepts}, system \eqref{eq:lti_virtual3} is output controllable on $[0,T]$ if and only if the output controllability Gramian
\begin{align}
    W(p,T) := C\widetilde{W}(p,T)C^\top\in \mathbb{R}^{r\times r} \label{def_output_con_Gra}
\end{align}
is nonsingular; equivalently, $\operatorname{rank} W(p,T) = r$. 
Here, $\widetilde{W}(p,T) \in \mathbb{R}^{n\times n}$ is the usual controllability Gramian of system \eqref{eq:lti_virtual3}.
In other words, $W(p,T)$ is the output Gramian obtained by projecting the state controllability Gramian $\widetilde{W}(p,T)$ through the output matrix $C$.
Moreover, $\widetilde{W}(p,T)$ is given by
\begin{align}
    \widetilde{W}(p,T) &= \int_0^T \exp(At) B(p)B(p)^\top \exp(A^\top t)\, {\rm d} t \\
    &= \sum_{i=1}^m p_i \widetilde{W}_i(T), \label{Def_Wc}
\end{align}
where, for each \(i\in\{1,\ldots,m\}\), \(\widetilde W_i(T)\) denotes the state controllability Gramian associated with the \(i\)th candidate input direction \(b_i\), that is,
\begin{align}
\label{eq:finite_time_horizon_gramian}
    \widetilde{W}_i(T) := \int_0^T \exp(At) b_i b_i^\top \exp(A^\top t)\, \D t.
\end{align} 
From \eqref{def_output_con_Gra} and \eqref{Def_Wc}, we obtain
\begin{align}
    W(p,T) = \sum_{i=1}^m p_i W_i(T), \label{eq_W(p,T)}
\end{align}
where $W_i(T)$ denotes the corresponding output controllability Gramian, that is,
\begin{align}
    W_i(T) := C \widetilde{W}_i(T) C^\top. \label{output_con_Gra}
\end{align}

Equivalently, \cite[Theorem III]{kreindler1964concepts} shows that system \eqref{eq:lti_virtual3} is output controllable on $[0,T]$ if and only if the output controllability matrix
\begin{align}
    \begin{pmatrix}
        CB(p) & CA B(p) & \cdots & CA^{n-1}B(p)
    \end{pmatrix} \label{output_con_rank}
\end{align}
has full row rank $r$. 
Thus, the output controllability property of system \eqref{eq:lti_virtual3} is independent of $T$.
For this reason, we simply say that system \eqref{eq:lti_virtual3} is output controllable, without explicitly referring to the horizon $[0,T]$.

%%%%%%%%%%%%%%%%%
\subsection{Geometric and Energy-Based Interpretation of Output Controllability} \label{subsec_geometric_energy}

In the classical setting, the controllability Gramian $\widetilde{W}(p,T)$ admits natural geometric and energy-based interpretations. 
Specifically, it characterizes the volume of the ellipsoid
\begin{align*}
    \widetilde{\mathcal{E}}(p,T) := \{x\in\Real^n \mid x^\top \widetilde{W}(p,T)^{-1} x \leq 1\},
\end{align*}
which represents the set of states that can be reached within unit control energy, 
and it quantifies the average input energy required to steer the state to points on the unit sphere in $\Real^n$. 

In the target setting, however, our concern is not the entire state $x(t)$ but only the designated target nodes specified by the output matrix $C$. 
Accordingly, the relevant object is the output controllability Gramian $W(p,T)$ defined in \eqref{def_output_con_Gra}. 
This Gramian determines the shape of the ellipsoid
\[
\mathcal{E}(p,T) := \{y\in\Real^r \mid y^\top W(p,T)^{-1} y \leq 1\},
\]
which represents the set of target outputs reachable within unit energy. 
The square root of $\det W(p,T)$ is proportional to the volume of $\mathcal{E}(p,T)$, and thus quantifies how widely the output can be driven with unit energy, analogously to \cite{muller1972analysis}.
Moreover, ${\rm tr}(W(p,T)^{-1})$ is proportional to the average of the minimum input energy $y_T^\top W(p,T)^{-1} y_T$ required to steer the output from the origin to $y_T$ uniformly distributed on the unit sphere in $\Real^r$.
Hence, $\det W(p,T)$ and ${\rm tr}(W(p,T)^{-1})$ respectively capture the volume of the reachable output set and the average input energy required to reach target outputs.

These geometric and energy-based viewpoints also motivate the choice of objective functions used below.
In particular, maximizing $\det W(p,T)$ corresponds to enlarging the reachable output ellipsoid $\mathcal{E}(p,T)$, whereas minimizing $\tr (W(p,T)^{-1})$ corresponds to reducing the average minimum input energy required to steer target outputs.

%%%%%%%%%%%%%%%%%
\subsection{Controllability Allocation Scores}
\label{Sec_CAS}

For any \(T>0\), define
\begin{align}
    \mathcal{X}_T
    :=
    \{p\in\mathbb{R}^m \mid W(p,T)\succ O\},
    \label{eq:XT}
\end{align}
which is an open and convex subset of \(\mathbb{R}^m\).
By Section~\ref{Sec_output_con}, \(p\in\mathcal{X}_T\) if and only if
system~\eqref{eq:lti_virtual3} is output controllable under the
allocation \(p\). 

On the feasible set \(\mathcal{X}_T\cap\Pset\),
we define the CASs through the following optimization problem for a given \(T>0\):
\begin{framed}
\vspace{-1em}
\begin{align}
\label{prob:unstable}
    \begin{aligned}
        &&& \text{minimize} && h_T(p) \\
        &&& \text{subject to} && p \in \mathcal{X}_T\cap\Pset.
    \end{aligned}
\end{align}
\vspace{-1em}
\end{framed}

\noindent
Here, motivated by the geometric and energy-based interpretations in
Section~\ref{subsec_geometric_energy}, \(h_T\) denotes either \(f_T\) or
\(g_T\) on \(\mathcal{X}_T\), defined as
\begin{align*}
f_T(p) := -\log\det W(p,T),\qquad
g_T(p) := \tr\paren{W(p,T)^{-1}}.
\end{align*}
The logarithm in \(f_T(p)\) improves numerical stability, since
\(\det W(p,T)\)---the product of the eigenvalues of \(W(p,T)\)---can
become numerically close to zero when several eigenvalues are extremely
small. The negative sign allows both \(f_T(p)\) and \(g_T(p)\) to be
formulated as minimization problems in a unified way.

We define the volumetric controllability allocation score (V-CAS) as
an optimal solution to problem~\eqref{prob:unstable} with \(h_T=f_T\),
and the average-energy controllability allocation score (AE-CAS) as an
optimal solution with \(h_T=g_T\).
The two proposed scores are defined on the same feasible set
\(\mathcal{X}_T\cap\Pset\) and are both based on the same output
controllability Gramian \(W(p,T)\), but they use different optimality
criteria. From the perspective in
Section~\ref{subsec_geometric_energy}, the V-CAS and AE-CAS respectively
capture the roles of the candidate input directions in enlarging
\(\mathcal{E}(p,T)\) and in reducing the control energy required to reach
the unit sphere in \(\mathbb{R}^r\).
Thus, the CAS framework evaluates the allocation of intervention
intensities among the candidate input directions with respect to the
target variables specified by \(C\).

For a general choice of the candidate input directions \(B_0\), the target map \(C\), and the allocation set \(\Pset\), the existence of an admissible allocation that renders the system output controllable is not guaranteed. Accordingly, throughout the remainder of this paper, we impose the following assumption.

\begin{assumption}\label{ass:output_controllable_allocation}
For the horizon \(T>0\) under consideration,
\begin{align}
    \mathcal{X}_T\cap\Pset\neq\emptyset.
\end{align}
\end{assumption}

Assumption~\ref{ass:output_controllable_allocation} ensures that at least one admissible allocation of the candidate input intensities renders system~\eqref{eq:lti_virtual3} output controllable. As shown below, this assumption is automatically satisfied in the TCS specialization.

%%%%%%%%%%%%%%%%%%%%%
\subsection{Target Controllability Scores}
\label{Sec_TCS_specialization}

The target controllability score (TCS) is obtained as a nodewise
specialization of the CAS framework in which the candidate input
directions and the target variables correspond to the same set of
network nodes. Let
\[
\mathcal{S}
=
\{s_1,\ldots,s_m\}
\subseteq
\{1,\ldots,n\}
\]
denote the target node set, and choose
\begin{align}
    B_0
    &:=
    \begin{pmatrix}
        e_{s_1}^{(n)} & \cdots & e_{s_m}^{(n)}
    \end{pmatrix},
    \label{eq:TCS_B0}\\
    C
    &:=
    B_0^\top
    =
    \begin{pmatrix}
        (e_{s_1}^{(n)})^\top\\
        \vdots\\
        (e_{s_m}^{(n)})^\top
    \end{pmatrix},
    \qquad
    \Pset=\Delta_m.
    \label{eq:TCS_specialization}
\end{align}
Then \(C\) selects the states of the target nodes, while the columns
of \(B_0\) apply the virtual inputs to those same nodes. In particular,
\(CB_0=I_m\), and hence
\begin{align}
    CB(p)
    =
    \diag(\sqrt{p_1},\ldots,\sqrt{p_m}).
\end{align}
Therefore, for every \(p\in{\rm ri}\,\Delta_m\), the first block
\(CB(p)\) of the output controllability matrix~\eqref{output_con_rank}
already has full row rank \(m\). Thus, the corresponding virtual system
is output controllable for every \(T>0\). In particular,
\begin{align}
\frac{1}{m}\mathbf{1}\in\mathcal{X}_T\cap\Delta_m, \label{TCS_assump_hold}
\end{align}
so Assumption~\ref{ass:output_controllable_allocation} is automatically
satisfied.

Under the specialization
\eqref{eq:TCS_B0}--\eqref{eq:TCS_specialization}, the resulting CASs are referred to as \emph{target controllability scores} (TCSs). In particular, the V-CAS is called the \emph{target volumetric controllability score} (target VCS), whereas the AE-CAS is called the \emph{target average-energy controllability score} (target AECS). Since each candidate input direction is paired with the corresponding target node in this specialization, the \(i\)th component of a TCS quantifies the relative importance of intervening at target node \(s_i\) under the corresponding output-controllability criterion.
If, in addition, \(m=n\) and
\(\mathcal{S}=\{1,\ldots,n\}\), then \(B_0=C=I_n\), and the TCS
reduces to the standard CS introduced in \cite{sato2022controllability}.
Table~\ref{tab:score_models} summarizes the relationship among standard CS, TCS, and the general CAS framework.

\begin{table*}[t]
\centering\small
\caption{Settings and interpretations of standard CS, TCS, and CAS. Here \(m\) denotes the number of candidate input directions and \(r\) the target-output dimension; \(m=r=n\) for standard CS and \(m=r=|\mathcal S|\) for TCS.}
\label{tab:score_models}
\renewcommand{\arraystretch}{1.18}
\begin{tabularx}{\textwidth}{@{}
>{\raggedright\arraybackslash}p{.10\textwidth}
>{\raggedright\arraybackslash}p{.16\textwidth}
>{\raggedright\arraybackslash}p{.16\textwidth}
>{\raggedright\arraybackslash}p{.13\textwidth}
>{\raggedright\arraybackslash}X@{}}
\toprule
Score & Candidate directions \(B_0\) & Target map \(C\) & Allocation set & Meaning and scope\\
\midrule
Standard CS
& \(I_n\)
& \(I_n\)
& \(\Delta_n\)
& Nodewise allocation priorities for controllability of the entire state.\\

TCS
& \(\begin{pmatrix}
        e_{s_1}^{(n)} & \cdots & e_{s_m}^{(n)}
    \end{pmatrix}\)
& \(B_0^\top\)
& \(\Delta_m\)
& Allocation priorities among the target nodes when the candidate input directions and target variables correspond to the same node set.\\

General CAS
& Prescribed \([b_1\ \cdots\ b_m]\)
& Prescribed full-row-rank \(C\)
& Compact convex \(\Pset\subseteq\Delta_m\)
& Allocation priorities among prescribed candidate input directions with respect to the target variables specified by \(C\).\\
\bottomrule
\end{tabularx}
\end{table*}

%%%%%%%%
\section{Computation and Uniqueness Analysis of
Controllability Allocation Scores}
\label{Sec_CAS_computation}

In this section, we present a projected-gradient method
for computing the CASs and establish conditions for uniqueness
of the optimal allocations, including generic uniqueness with
respect to the time horizon.
We also examine qualitative differences between the TCS
and the standard full-state controllability scores.

\subsection{Projected-Gradient Computation}
\label{subsec3-A}

The CASs can be computed using Algorithm~\ref{alg:projgrad}, which is
based primarily on the domain-aware projected-gradient framework developed
in \cite[Section~II]{sato2026projected}. We adapt this framework to the present
output-controllability setting, with the explicit constraint set \(\Pset\)
and the open objective domain \(\mathcal{X}_T\).
This framework extends the Armijo projected-gradient approach used for
nodewise controllability scoring in \cite{sato2022controllability}.

\begin{figure}[!t]
\begin{algorithm}[H]
   \caption{A projected gradient method}
   \label{alg:projgrad}
 \textbf{Input:} Output Controllability Gramians
 $W_1(T),\ldots, W_m(T)$ in \eqref{output_con_Gra},
 $p^{(0)} \in \mathcal{X}_T \cap \Pset$, and
 $\varepsilon\geq 0$.\\
 \textbf{Output:} CAS.
   \begin{algorithmic}[1]
   \FOR{$k=0,1,\ldots$}
   \STATE $q^{(k)} := p^{(k)}-
   \alpha^{(k)} \nabla h_T(p^{(k)})$, where $\alpha^{(k)}$ is
   defined by using Algorithm~\ref{alg:Armijo}.
   \STATE $p^{(k+1)}
   := \Pi_{\Pset}(q^{(k)})$.

       \IF{$\|p^{(k)}-p^{(k+1)}\|\leq \varepsilon$}
       \RETURN $p^{(k+1)}$.
       \ENDIF
   \ENDFOR
   \end{algorithmic}
\end{algorithm}
\end{figure}

\begin{figure}[!t]
\begin{algorithm}[H]
   \caption{Domain-aware Armijo rule along the projection arc}
   \label{alg:Armijo}
   \textbf{Input:} $\sigma,\rho\in (0,1)$ and $\alpha>0$.\\
   \textbf{Output:} Step size.
   \begin{algorithmic}[1]
   \STATE $\tilde{p}^{(k)}
   :=\Pi_{\Pset}(p^{(k)}-\alpha \nabla h_T(p^{(k)}))$.
   \IF{$\tilde{p}^{(k)}\in\mathcal{X}_T$}
       \IF{$h_T(\tilde{p}^{(k)})\leq h_T(p^{(k)})
       + \sigma \nabla h_T(p^{(k)})^\top
       (\tilde{p}^{(k)}-p^{(k)})$}
           \RETURN $\alpha^{(k)}:=\alpha$.
       \ENDIF
   \ENDIF
   \STATE $\alpha\leftarrow \rho \alpha$, and go back to step 1.
   \end{algorithmic}
\end{algorithm}
\end{figure}

Since the intersection $\mathcal{X}_T \cap \Pset$ is not closed in general, as demonstrated in \cite[Section~III-A]{sato2022controllability}, the existence of an optimal solution to problem \eqref{prob:unstable} is not immediate.
Nevertheless, it can be established as follows:
Define the sublevel set
\begin{align*}
\mathcal{F}_T^{(0)}
:= \{p\in {\bb R}^m \mid h_T(p)\leq h_T(p^{(0)})\}
\cap (\mathcal{X}_T\cap \Pset),
\end{align*}
where $p^{(0)}$ is any fixed point in
$\mathcal{X}_T\cap\Pset$.
Arguing as in the proof of \cite[Lemma~1]{sato2022controllability}, we can regard
\(\mathcal{F}_T^{(0)}\) as the sublevel set of \(h_T\) over
\(\Pset\), namely,
\begin{align}
\mathcal{F}_T^{(0)}
= \{p\in \tblue{\Pset} \mid h_T(p)\leq h_T(p^{(0)})\}.
\label{eq_local_feasible}
\end{align}
Indeed, no sequence in \(\mathcal{F}_T^{(0)}\) can converge to a point in
\(\Pset\setminus \mathcal{X}_T\), since \(h_T(p)\to+\infty\) as
\(W(p,T)\) approaches singularity.

Consequently, problem~\eqref{prob:unstable} is equivalent to
\begin{align}
   \begin{aligned}
       &&& \text{minimize} && h_T(p) \\
       &&& \text{subject to} && p\in \mathcal{F}_T^{(0)}.
   \end{aligned}
   \label{problem2}
\end{align}
Since $\Pset$ is compact and $h_T$ is continuous on
$\mathcal{X}_T$, the set
$\mathcal{F}_T^{(0)}$ is a nonempty compact subset of ${\bb R}^m$.
Thus, by the Weierstrass theorem, problem~\eqref{problem2} has an optimal
solution (see \cite[Proposition~A.8]{bertsekas2016nonlinear}), and
problem~\eqref{prob:unstable} also admits an optimal solution.

Thus, problem~\eqref{prob:unstable} admits optimal solutions, that is, the V-CAS and AE-CAS.
The two CASs are based on the same output
controllability Gramian \(W(p,T)\), but they use different optimality
criteria.
From the perspective in Section~\ref{subsec_geometric_energy}, the
V-CAS and AE-CAS respectively capture the roles
of the candidate input directions in enlarging
\(\mathcal{E}(p,T)\) and in reducing the control energy to reach the unit
sphere in \(\Real^r\).

If an optimal solution to problem \eqref{prob:unstable} is unique,
the corresponding CAS provides a uniquely determined allocation
score for the candidate input directions.
Indeed, the gradients of $f_T(p)$ and $g_T(p)$ are given by
\begin{align}
(\nabla f_T(p))_i &= -\tr(W(p,T)^{-1}W_i(T)), \\
(\nabla g_T(p))_i &= -\tr(W(p,T)^{-1}W_i(T)W(p,T)^{-1})
\end{align}
for $i=1,\dots,m$.
Since $W_i(T)\succeq O$, we have
\begin{align}
(\nabla f_T(p))_i \leq 0, \quad
(\nabla g_T(p))_i \leq 0
\end{align}
for any $p\in \mathcal{X}_T$.
Moreover, both inequalities are strict whenever $W_i(T)\neq O$.
Thus, before accounting for the coupling imposed by the allocation
constraint \(\Pset\), increasing the intensity assigned to a nontrivial
candidate input direction improves both controllability criteria.

The CASs can be calculated using Algorithm~\ref{alg:projgrad}, which
is a modification of the method proposed in
\cite{sato2022controllability}.
There are two key differences:
\begin{itemize}
    \item While \cite{sato2022controllability} considered the case
    $m=n$, Algorithm~\ref{alg:projgrad}
    allows prescribed candidate input directions and a general
    compact convex allocation set \(\Pset\subseteq\Delta_m\).
    \item Algorithm~\ref{alg:projgrad} is based on the output
    controllability Gramians \eqref{output_con_Gra}, whereas
    \cite{sato2022controllability} used the standard controllability
    Gramians \eqref{eq:finite_time_horizon_gramian}.
\end{itemize}

In Algorithm~\ref{alg:projgrad},
$\Pi_{\Pset}$ in step~3 denotes the Euclidean projection onto
$\Pset$.
In the TCS specialization \(\Pset=\Delta_m\), this projection can be
computed efficiently as detailed in \cite{condat2016fast}.
Here, the projection is taken only onto $\Pset$, not onto
$\mathcal{X}_T$, because $\mathcal{X}_T$ is an open set in ${\bb R}^m$
and thus a nearest-point projection onto $\mathcal{X}_T$ is not
guaranteed to exist uniquely for all points.
Since \(h_T\) is defined only on \(\mathcal{X}_T\), the Armijo
backtracking procedure in Algorithm~\ref{alg:Armijo} first checks whether
the projected trial point belongs to \(\mathcal{X}_T\) before evaluating
\(h_T\). This domain-aware projection-arc procedure separates projection
onto the compact convex set \(\Pset\) from membership in the open domain
\(\mathcal{X}_T\). By the same open-domain argument as in
\cite[Theorem~4]{sato2026projected}, the backtracking procedure is
well-defined, the accepted iterates remain in the compact initial
sublevel set \(\mathcal{F}_T^{(0)}\), and the generated sequence
converges to an optimal solution.

\begin{proposition}
Let $\{p^{(k)}\}$ be a sequence generated by
Algorithm~\ref{alg:projgrad} with $\varepsilon=0$.
Then, there exists an optimal solution $p^*$ to problem
\eqref{prob:unstable} such that
\begin{align*}
    \lim_{k\rightarrow \infty} p^{(k)} = p^*.
\end{align*}
\end{proposition}

We show in Section~\ref{subsec:CAS_uniqueness} that
problem~\eqref{prob:unstable} admits a unique optimal solution
under suitable conditions on the candidate Gramian contributions and
the allocation set \(\Pset\).
Therefore, whenever these uniqueness conditions hold, the sequence
$\{p^{(k)}\}$ converges to the uniquely determined CAS.

\begin{remark}
The analysis in
\cite{sato2026projected} further shows that linear convergence holds
under a positive-curvature condition along the feasible displacement
directions.
\end{remark}

\subsection{Uniqueness of Optimal Solutions}
\label{subsec:CAS_uniqueness}

We now turn to the uniqueness of the CASs, that is, the optimal
solutions to problem~\eqref{prob:unstable}. Such uniqueness is important
when a CAS is interpreted as a uniquely determined allocation profile
among the candidate input directions.

To this end, we use the Hessians
of $f_T(p)$ and $g_T(p)$, which are expressed as
\begin{align}
    &(\nabla^2 f_T(p))_{ij} \nonumber\\
    &= \tr (W(p,T)^{-1} W_i(T)W(p,T)^{-1}W_j(T)), \label{Hessian_f}\\
     & (\nabla^2 g_T(p))_{ij} \label{Hessian_g} \\
     &= \tr (W(p,T)^{-1} W_i(T)W(p,T)^{-1}W_j(T)W(p,T)^{-1}) \nonumber\\
      &\,\,\,\, +\tr (W(p,T)^{-1} W_j(T)W(p,T)^{-1}W_i(T)W(p,T)^{-1}), \nonumber
\end{align}
respectively.
For the uniqueness analysis, it is useful to rewrite these Hessians in quadratic form.
The following lemma provides such a representation.

\begin{lemma} \label{Lem_nijikeishiki}
Suppose that $T>0$ and $p\in \mathcal{X}_T$ are given. Then,
for any $x\in {\bb R}^m$,
\begin{align}
x^\top \nabla^2 f_T(p) x
&= {\rm tr} \left(G(p,x,T)^2 \right), \label{f_nijikeishiki}\\
x^\top \nabla^2 g_T(p) x
&= 2{\rm tr}(W(p,T)^{-1} G(p,x,T)^2), \label{g_nijikeishiki}
\end{align}
where
$G(p,x,T):= W(p,T)^{-1/2}W(x,T)W(p,T)^{-1/2}$
with \(W(x,T):=\sum_{i=1}^m x_iW_i(T)\).
\end{lemma}

\begin{proof}
It follows from \eqref{Hessian_f} that for any $x\in {\bb R}^m$,
\begin{align*}
x^\top \nabla^2 f_T(p) x
&= \sum_{i,j=1}^m x_i (\nabla^2 f_T(p))_{ij} x_j \\
&= {\rm tr} (W(p,T)^{-1} W(x,T) W(p,T)^{-1} W(x,T)).
\end{align*}
Thus, \eqref{f_nijikeishiki} holds.
Similarly, it follows from \eqref{Hessian_g} that
  \begin{align*}
    &x^\top \nabla^2 g_T(p) x = \sum_{i,j=1}^m x_i (\nabla^2 g_T(p))_{ij} x_j \\
    =& 2{\rm tr} (W(p,T)^{-1} W(x,T) W(p,T)^{-1} W(x,T) W(p,T)^{-1}).
    \end{align*}
Hence, \eqref{g_nijikeishiki} holds. \qed
\end{proof}

To account for the general allocation set \(\Pset\), define the subspace
\begin{align}
\mathcal{L}_{\Pset}
:=
{\rm lin}(\Pset-\Pset)
\label{eq:LP}
\end{align}
and the linear map
\begin{align}
\Phi_T(x)
:=
\sum_{i=1}^m x_i W_i(T).
\label{eq:PhiT}
\end{align}
Only directions in \(\mathcal{L}_{\Pset}\) can occur as differences
between admissible allocations. Therefore, injectivity of \(\Phi_T\)
on $\mathcal{L}_{\Pset}$, rather than on all of \(\mathbb{R}^m\), is the
relevant condition for uniqueness.
Define the optimal-allocation sets
\begin{align}
\mathcal S_V(T)
=
\argmin_{p\in\mathcal X_T\cap\Pset} f_T(p),\quad
\mathcal S_A(T)
=
\argmin_{p\in\mathcal X_T\cap\Pset} g_T(p).
\label{eq:exact_solutions}
\end{align}

Lemma~\ref{Lem_nijikeishiki} yields the following characterization
of the optimal-allocation sets and a sufficient condition for uniqueness.

\begin{lemma}
\label{thm:fixed_horizon_uniqueness}
Suppose that Assumption~\ref{ass:output_controllable_allocation} holds,
and let \(T>0\) be fixed.
For each \(\bullet\in\{V,A\}\), choose any
\(p_\bullet^*\in\mathcal S_\bullet(T)\).
Then
\begin{align}
\mathcal{S}_\bullet(T)
=
\left\{
p\in\Pset
\;\middle|\;
W(p,T)=W(p_\bullet^*,T)
\right\}.
\label{eq:optimal_fiber}
\end{align}

If \(\mathcal{L}_{\Pset}=\{0\}\), then \(\Pset\) is a singleton and
both CASs are unique. Otherwise, if
\begin{align}
{\rm Ker}\,\Phi_T\cap\mathcal{L}_{\Pset}
=
\{0\},
\label{unique_sufficient}
\end{align}
then the V-CAS and AE-CAS are unique.
\end{lemma}

\begin{proof}
We prove \eqref{eq:optimal_fiber} for $\bullet=V$.
The function
$W\mapsto-\log\det W$
is strictly convex on \(\mathbb{S}_{++}^r\).
Suppose that \(p_1,p_2\in\mathcal S_V(T)\) satisfy
$W(p_1,T)\neq W(p_2,T)$.
Since \(\Pset\) is convex and \(W(p,T)\) depends linearly on \(p\),
their midpoint belongs to \(\Pset\), and
$W\left(\frac{p_1+p_2}{2},T\right)
=
\frac{W(p_1,T)+W(p_2,T)}{2}$.
Because \(p_1,p_2\in\mathcal S_V(T)\), both endpoint Gramians are
positive definite, and hence so is their midpoint.
Strict convexity therefore gives
$f_T\left(\frac{p_1+p_2}{2}\right)
<
\frac{f_T(p_1)+f_T(p_2)}{2}$,
which is strictly smaller than the optimal value, a contradiction.
Thus, all allocations in \(\mathcal S_V(T)\) produce the same
optimal output controllability Gramian.
Now fix any \(p_V^*\in\mathcal S_V(T)\).
If \(p\in\Pset\) satisfies
$W(p,T)=W(p_V^*,T)$,
then \(W(p,T)\succ O\), so \(p\in\mathcal X_T\cap\Pset\), and
$f_T(p)=f_T(p_V^*)$.
Hence \(p\in\mathcal S_V(T)\).
Together with the preceding argument, this proves
\eqref{eq:optimal_fiber} for $\bullet=V$.

By the same argument, 
for any \(p_A^*\in\mathcal S_A(T)\),
\eqref{eq:optimal_fiber} holds for $\bullet=A$,
because
$W\mapsto\tr(W^{-1})$
is strictly convex on \(\mathbb{S}_{++}^r\).

We next consider uniqueness of the CASs.
If \(\mathcal L_{\Pset}=\{0\}\), then every difference of two points
in \(\Pset\) is zero, so \(\Pset\) is a singleton and both CASs are
unique.
Suppose now that \(\mathcal L_{\Pset}\neq\{0\}\) and that
\eqref{unique_sufficient} holds.
If \(p_1,p_2\in\Pset\) satisfy
$W(p_1,T)=W(p_2,T)$,
then
\[
0
=
W(p_1,T)-W(p_2,T)
=
\Phi_T(p_1-p_2).
\]
Since \(p_1-p_2\in\mathcal L_{\Pset}\), condition
\eqref{unique_sufficient} implies \(p_1-p_2=0\).
Therefore, \(\Phi_T\) is injective on the set $\mathcal{L}_{\Pset}$. Combining this fact with
\eqref{eq:optimal_fiber} shows that each of
\(\mathcal S_V(T)\) and \(\mathcal S_A(T)\) is a singleton.
Hence, the V-CAS and AE-CAS are unique.
\qed
\end{proof}

The restricted-injectivity condition
\eqref{unique_sufficient} can be checked through a matrix rank test.
In fact, define
\begin{align}
G(T)
&:=
\begin{pmatrix}
{\rm svec}(W_1(T)) &
\cdots &
{\rm svec}(W_m(T))
\end{pmatrix},
\label{def_G}\\
K(T)
&:=
G(T)U_{\Pset},
\label{def_K}
\end{align}
where the columns of \(U_{\Pset}\) form an orthonormal basis of
\(\mathcal{L}_{\Pset}\).
The matrix \(K(T)\) describes how differences between
admissible allocations are reflected in the output
controllability Gramian.
In fact, for any \(p,p'\in\Pset\), we have $p-p'\in \mathcal{L}_{\Pset}$, and thus
\(p-p'=U_{\Pset}z\) for some \(z\in\mathbb R^\ell\), where
\(\ell=\dim\mathcal L_{\Pset}\), yielding
\begin{align}
K(T)z
=
\operatorname{svec}
\bigl(W(p,T)-W(p',T)\bigr).
\end{align}

\begin{lemma}
\label{lem:KT_restricted_injectivity}
Under the above setting,
\begin{align}
K(T) \text{ has full column rank}
\quad\Leftrightarrow\quad
\eqref{unique_sufficient} \text{ holds}.
\end{align}
\end{lemma}

\begin{proof}
For any \(z\in\mathbb R^\ell\), where
\(\ell=\dim\mathcal L_{\Pset}\), we have
\begin{align}
K(T)z
&=
\operatorname{svec}
\left(
\sum_{i=1}^m
(U_{\Pset}z)_iW_i(T)
\right)
=
\operatorname{svec}
\left(
\Phi_T(U_{\Pset}z)
\right),
\end{align}
and thus
\[
K(T)z=0
\quad\Leftrightarrow\quad
\Phi_T(U_{\Pset}z)=O.
\]
Since the columns of \(U_{\Pset}\) form a basis of
\(\mathcal L_{\Pset}\), every
\(x\in\mathcal L_{\Pset}\) can be uniquely expressed as
\(x=U_{\Pset}z\) for some \(z\in\mathbb R^\ell\).
Moreover, \(z=0\) if and only if \(U_{\Pset}z=0\).
Hence, the result follows.
\qed
\end{proof}

Combining Lemmas~\ref{thm:fixed_horizon_uniqueness} and
\ref{lem:KT_restricted_injectivity}, full column rank of \(K(T)\)
guarantees uniqueness of both the V-CAS and AE-CAS.
We next show that this rank condition holds for almost all
\(T>0\) if it holds at some \(T_0>0\).

\begin{theorem}
\label{thm:generic_uniqueness}
Suppose that Assumption~\ref{ass:output_controllable_allocation} holds.
If \(K(T_0)\) has full column rank for some \(T_0>0\), then
V-CAS and AE-CAS are unique for almost all \(T>0\).
\end{theorem}

\begin{proof}
Each entry of \(W_i(T)\), and hence each entry of \(K(T)\), is real
analytic in \(T>0\).
If \(K(T_0)\) has full column rank for some \(T_0>0\), there exists a
full-size minor of \(K(T)\) that is nonzero at \(T_0\).
Since this minor is a nonzero real-analytic function of \(T\), its
zeros are isolated \cite{mityagin2015zero}.
Therefore, \(K(T)\) can lose full column rank only on a discrete subset
of \((0,\infty)\).
By Lemmas~\ref{thm:fixed_horizon_uniqueness} and
\ref{lem:KT_restricted_injectivity}, the CAS is unique at
every horizon outside this set. \qed
\end{proof}

Theorem~\ref{thm:generic_uniqueness} shows that a single
successful rank test is sufficient to certify uniqueness
for almost all positive horizons.
Although the rank condition may fail at isolated horizons,
such rank deficiency does not necessarily imply nonuniqueness
of the CASs.

To apply Theorem~\ref{thm:generic_uniqueness}, one must identify
a horizon \(T_0>0\) at which \(K(T_0)\) has full column rank.
The theorem itself does not guarantee the existence of such
a horizon for arbitrary choices of \(A\), \(B_0\), \(C\), and
\(\Pset\).

To obtain a sufficient condition that does not require
checking the rank of \(K(T)\) at a particular horizon,
we consider the short-time behavior of the candidate output
controllability Gramians.
For \(q_i:=Cb_i\), \(i=1,\ldots,m\), we have
\begin{align}
\lim_{T\downarrow0}T^{-1}W_i(T)=q_iq_i^\top. \label{W_ilimit}
\end{align}
Motivated by this limit, define
\begin{align}
G_{\mathrm{inst}}
&:=
\begin{pmatrix}
\operatorname{svec}(q_1q_1^\top) &
\cdots &
\operatorname{svec}(q_mq_m^\top)
\end{pmatrix},
\label{G_inst}\\
K_{\mathrm{inst}}
&:=G_{\mathrm{inst}}U_{\Pset}.
\label{eq:instantaneous_rank}
\end{align}

\begin{theorem}
\label{thm:generic_uniqueness2}
If \(K_{\mathrm{inst}}\) has full column rank,
then for any \(A\in\mathbb R^{n\times n}\) satisfying
Assumption~\ref{ass:output_controllable_allocation},
V-CAS and AE-CAS are unique
for all sufficiently small \(T>0\) and for almost all \(T>0\).
\end{theorem}
\begin{proof}
By \eqref{def_G} and the linearity of
\(\operatorname{svec}\), we have
\begin{align}
T^{-1}G(T)
=
\begin{pmatrix}
\operatorname{svec}(T^{-1}W_1(T)) &
\cdots &
\operatorname{svec}(T^{-1}W_m(T))
\end{pmatrix}.
\end{align}
Since \(\operatorname{svec}\) is continuous,
\eqref{W_ilimit} and \eqref{G_inst} imply that
\begin{align}
\lim_{T\downarrow0}T^{-1}G(T)
&=
\begin{pmatrix}
\operatorname{svec}(q_1q_1^\top) &
\cdots &
\operatorname{svec}(q_mq_m^\top)
\end{pmatrix}
=
G_{\mathrm{inst}}.
\end{align}
Consequently, \eqref{def_K} and \eqref{eq:instantaneous_rank} yield
\begin{align}
\lim_{T\downarrow0}T^{-1}K(T)
=
G_{\mathrm{inst}}U_{\Pset}
=
K_{\mathrm{inst}}.
\label{eq:K_inst_limit}
\end{align}

Since \(K_{\mathrm{inst}}\) has full column rank,
its smallest singular value is positive.
By \eqref{eq:K_inst_limit} and the continuity
of singular values~\cite[Theorem 2.6.4]{horn2012matrix},
\(K(T)\) has full
column rank for all sufficiently small \(T>0\).
By Lemmas~\ref{thm:fixed_horizon_uniqueness} and
\ref{lem:KT_restricted_injectivity},
the V-CAS and AE-CAS are unique for all such horizons.
Theorem~\ref{thm:generic_uniqueness} then guarantees
uniqueness for almost all \(T>0\).
\qed
\end{proof}

The following example illustrates that the rank condition in
Theorem~\ref{thm:generic_uniqueness2} can hold even when
the limiting matrices in \eqref{W_ilimit} are linearly dependent.

\begin{example}[Multiple intervention directions for two target variables]
\label{ex:multiple_group_interventions}
Let \(C\in\mathbb R^{2\times n}\) be a full-row-rank matrix
specifying two target variables, such as the average
activities of two network groups.
Choose \(a_1,a_2\in\mathbb R^n\) such that
\(Ca_i=e_i^{(2)}\), \(i=1,2\), and set
\begin{align}
B_0
=
\begin{pmatrix}
a_1 & a_2 & a_1+a_2 & a_1+2a_2
\end{pmatrix},
\qquad
\Pset=\Delta_4.
\end{align}
The four columns of \(B_0\) represent candidate
intervention directions acting on individual or
both target variables.
Then
\begin{align}
CB_0=
\begin{pmatrix}
1&0&1&1\\
0&1&1&2
\end{pmatrix}=: \begin{pmatrix}
    q_1 & q_2 & q_3 & q_4
\end{pmatrix}.
\end{align}
A direct calculation gives
$\Ker G_{\mathrm{inst}}
=
\operatorname{span}
\left\{
(1,-2,-2,1)^\top
\right\}$.
Since
\(\mathcal L_{\Delta_4}
=\{v\in\mathbb R^4:\mathbf1^\top v=0\}\),
we obtain
$\Ker G_{\mathrm{inst}}
\cap\mathcal L_{\Delta_4}
=\{0\}$.
Hence, \(K_{\mathrm{inst}}\) has full column rank.
Moreover, for \(p=\frac14\mathbf1\in\Delta_4\),
\begin{align}
CB(p)B(p)^\top C^\top
=
\frac14\sum_{i=1}^4q_iq_i^\top
\succeq\frac14 I_2\succ O.
\end{align}
Thus, Assumption~\ref{ass:output_controllable_allocation}
holds for every \(A\), and
Theorem~\ref{thm:generic_uniqueness2} guarantees
uniqueness of both CASs for all sufficiently small
positive horizons and almost all positive horizons.
In particular, although \(G_{\mathrm{inst}}\) cannot
have full column rank because \(m=4>\dim\mathbb S^2=3\),
\(K_{\mathrm{inst}}\) has full column rank.
\end{example}

The general rank condition in
Theorem~\ref{thm:generic_uniqueness2} admits a simpler sufficient
condition when the number of candidate input directions equals the
target-output dimension.

\begin{corollary}
\label{cor:generic_uniqueness2}
If
\begin{align}
m=r
\quad \text{and}\quad
CB_0 \ \text{is nonsingular},
\label{eq:instantaneous_nondegeneracy}
\end{align}
then for any \(A\in\mathbb R^{n\times n}\) satisfying
Assumption~\ref{ass:output_controllable_allocation}, the V-CAS and AE-CAS are unique for all sufficiently small \(T>0\) and for almost all \(T>0\).    
\end{corollary}
\begin{proof}
Set \(M:=CB_0\). Since \(m=r\), we have
\(q_i=Me_i^{(r)}\) for \(i=1,\ldots,r\).
Thus, for any \(z\in\mathbb R^\ell\), where
\(\ell=\dim\mathcal L_{\Pset}\), we have
\begin{align}
K_{\mathrm{inst}}z
&=
\operatorname{svec}\left(
\sum_{i=1}^r (U_{\Pset}z)_i q_iq_i^\top
\right)
\nonumber\\
&=
\operatorname{svec}\left(
M\diag(U_{\Pset}z)M^\top
\right).
\end{align}
Since \(M\) is nonsingular and
\(\operatorname{svec}\) is injective,
\[
K_{\mathrm{inst}}z=0
\quad\Longleftrightarrow\quad
U_{\Pset}z=0
\quad\Longleftrightarrow\quad
z=0,
\]
where the last equivalence follows from the
orthonormality of the columns of \(U_{\Pset}\).
Hence, \(K_{\mathrm{inst}}\) has full column rank,
and the conclusion follows from
Theorem~\ref{thm:generic_uniqueness2}.
\qed
\end{proof}

The following example illustrates a practically relevant class of
virtual-intervention models that satisfies
condition~\eqref{eq:instantaneous_nondegeneracy}.

\begin{example}[Group-level virtual intervention]
\label{ex:group_level_virtual_intervention}
Consider a network whose nodes are partitioned into \(r\) disjoint
groups. Let the \(i\)th target output be an aggregate state of the
\(i\)th group,
\begin{align}
y_i=c_i^\top x,
\qquad i=1,\ldots,r,
\end{align}
and define
$C
=
\begin{pmatrix}
c_1^\top\\
\vdots\\
c_r^\top
\end{pmatrix}$.
For example, \(c_i^\top x\) may represent the average activity of the
nodes in the \(i\)th group.

Suppose that the candidate directions
\(b_1,\ldots,b_r\) are introduced as virtual distributed
interventions associated with the corresponding groups and satisfy
\begin{align}
c_i^\top b_i &>0,
&
c_i^\top b_j &=0
\quad (i\neq j).
\label{eq:group_aligned_inputs}
\end{align}
Then, with
$B_0=
\begin{pmatrix}
b_1 & \cdots & b_r
\end{pmatrix}$,
we have
\begin{align}
CB_0
=
\diag\bigl(
c_1^\top b_1,\ldots,c_r^\top b_r
\bigr)
\succ O.
\end{align}
Hence,
condition~\eqref{eq:instantaneous_nondegeneracy} is satisfied.

In this setting, each candidate direction represents a virtual
intervention distributed over a group of nodes rather than an
intervention at a single node. The resulting CAS therefore quantifies
how intervention intensity is allocated among target groups.
The groups may represent, for example, communities, spatial regions,
or clusters, depending on the application.
\end{example}

Condition~\eqref{eq:instantaneous_nondegeneracy} in Corollary~\ref{cor:generic_uniqueness2} is also
automatically satisfied in the TCS specialization, because each
candidate input direction is paired with the corresponding target
coordinate. This yields the following uniqueness result for the
target VCS and target AECS.

\begin{corollary}
\label{cor:generic_uniqueness_TCS}
Under the TCS specialization
\eqref{eq:TCS_B0}--\eqref{eq:TCS_specialization}, for arbitrary $A\in {\bb R}^{n\times n}$,
the target VCS and target AECS are unique
for all sufficiently small \(T>0\) and for
almost all \(T>0\).
\end{corollary}

\begin{proof}
Under the TCS specialization,
\(r=m\) and
$CB_0
=
B_0^\top B_0
=
I_m$,
because the columns of \(B_0\) are distinct standard basis vectors.
Hence, condition~\eqref{eq:instantaneous_nondegeneracy} is satisfied.
Moreover, Assumption~\ref{ass:output_controllable_allocation} holds from \eqref{TCS_assump_hold}.
The result therefore follows directly from
Corollary~\ref{cor:generic_uniqueness2}.
\qed
\end{proof}

Corollary~\ref{cor:generic_uniqueness_TCS} justifies interpreting the
TCS as a node-centrality measure, since uniqueness makes the nodewise
scores unambiguously comparable across target nodes and thereby yields
a uniquely determined ranking for almost all horizons.

Note that this uniqueness is less critical for a general CAS. Since \(B_0\) and
\(C\) are prescribed separately, a CAS is more naturally interpreted
as an optimal allocation among candidate intervention directions than
as a node-centrality measure. Thus, even when the optimal allocation is
not unique, any optimal allocation may still be useful for intervention
design.

\begin{remark}[Detecting nonuniqueness]
For a fixed \(T>0\), the rank condition in Lemma~3
is sufficient but not necessary for uniqueness.
By Lemma~2, for any CAS \(p_\bullet^*\),
\begin{align}
\mathcal S_\bullet(T)
=
\{p\in\Pset:G(T)p=G(T)p_\bullet^*\},
\quad \bullet\in\{V,A\}.
\end{align}
Thus, the corresponding CAS is unique if and only if
the minimum and maximum of \(p_j\) over this set
coincide for every \(j=1,\ldots,m\).
When \(\Pset\) is polyhedral, these extrema can be
computed by linear programming.
\end{remark}

%%%%%%%%%%%%%%%%%

\subsection{Differences from Standard VCS/AECS}
\label{sec:difference_standard_target}

Although the preceding results establish uniqueness of the
CASs and TCSs for almost all positive horizons, uniqueness alone does not imply that their qualitative
behavior resembles that of the standard VCS/AECS.
In particular, symmetry-based conclusions known for the standard setting
do not directly carry over to the target setting:
When $A$ is symmetric, the standard VCS becomes uniform, and when $A$ is
skew-symmetric, both the standard VCS and AECS become uniform, as shown in
\cite[Section~III-B]{sato2025uniqueness}.
These conclusions exploit the specific full-state, nodewise-input
structure of the standard controllability scoring.

By contrast, in the target setting considered in this paper, we evaluate
\eqref{eq_W(p,T)}, where each $W_i(T)$ is an output controllability
Gramian associated with the target nodes.
Hence, even if $A$ is symmetric or skew-symmetric, projection onto the
target nodes may break the nodewise symmetry that leads to uniform
optimal solutions in the standard full-state setting.
Therefore, one should not expect the target VCS/AECS to inherit the same
conclusions.
More generally, the CAS framework introduces additional dependence
on the prescribed candidate input directions \(B_0\), the target map
\(C\), and the allocation set \(\Pset\). Consequently, structural
properties of \(A\) alone do not, in general, determine the form of an
optimal CAS.

To illustrate the difference from the standard CSs, consider the
TCS specialization with \(n=3\), \(m=2\), \(T=1\),
\begin{align}
B_0
&=
\begin{pmatrix}
1&0\\
0&1\\
0&0
\end{pmatrix},
&
C&=B_0^\top,
&
\Pset&=\Delta_2.
\label{eq:TCS_examples_setting}
\end{align}
The following examples show that symmetry properties of \(A\) that
imply uniformity in the standard setting do not generally imply
uniformity of the corresponding TCSs.

\begin{example}[A symmetric matrix for which target VCS is not uniform]
Under the setting \eqref{eq:TCS_examples_setting}, consider
    $A=
    \begin{pmatrix}
        0&2&2\\
        2&0&1\\
        2&1&0
    \end{pmatrix}$,
which is symmetric.
The resulting target VCS is
\begin{align}
    p^{\rm VCS}\approx (0.3999,\,0.6001),
\end{align}
which is not the uniform allocation $(1/2,1/2)$.
Thus, even when $A$ is symmetric, the target VCS need not be uniform.
\end{example}

\begin{example}[A skew-symmetric matrix for which target VCS/AECS are not uniform]
Under the setting \eqref{eq:TCS_examples_setting}, consider
    $A=
    \begin{pmatrix}
        0&-3&0\\
        3&0&-1\\
        0&1&0
    \end{pmatrix}$,
which is skew-symmetric.
The resulting target VCS $p^{\rm VCS}$ and target AECS $p^{\rm AECS}$
are
\begin{align}
    p^{\rm VCS}\approx (0, 1),
    \qquad
    p^{\rm AECS}\approx (0, 1),
\end{align}
which are far from the uniform allocation $(1/2,1/2)$.
Therefore, unlike the standard setting, skew-symmetry of $A$ does not
imply uniformity of the target VCS/AECS.
\end{example}

The above examples show that the target controllability score is not
merely a straightforward restriction of the standard controllability
score to a subset of nodes.
In the target setting, projection onto the target nodes changes the
structure of the Gramian, and this can fundamentally alter the optimal
allocation.
Therefore, the qualitative behavior of target VCS/AECS should be analyzed
separately from that of the standard VCS/AECS.
The same observation applies more broadly to CAS, for which both
the target map and the candidate intervention directions affect the
candidate output Gramians and hence the optimal allocation.

At the same time, however, this distinction does not rule out the
possibility that a suitably constructed surrogate formulation may
provide an accurate approximation to the full-network target formulation.

%%%%%%%%%%%%%%
\section{Approximate computation of CAS and TCS}
\label{sec:efficient}

Computational efficiency is practically relevant when controllability
scores are used to inform intervention decisions under evolving
operating conditions. In such applications, the system matrix \(A\)
in \eqref{system0} may remain representative only over an
application-dependent time interval. The resulting scores should
therefore be computed sufficiently quickly to remain relevant to the
underlying dynamics, motivating efficient approximations with
quantitative accuracy guarantees.

To analyze such approximations, we compare the exact and surrogate
formulations through their candidate output controllability Gramians.
This framework accommodates reduced-order models, approximate Gramian
computations, and other constructions of surrogate candidate Gramians.
We first relate relative Gramian errors to the performance of surrogate
allocations under the original CAS objectives and clarify how these
errors can be controlled through candidate-Gramian discrepancies.
We then specialize the analysis to a target-only TCS surrogate,
derive explicit bounds on these discrepancies, and develop a
computable relative-error certificate.

\subsection{General Surrogate Candidate Gramians}
\label{sec:surrogate}

We define a surrogate problem by replacing the candidate Gramians
\eqref{output_con_Gra} with matrices
$\wh W_i(T)\in\mathbb S_+^r$
$(i=1,\ldots,m)$,
where each $\wh W_i(T)$ corresponds to the same candidate input
as $W_i(T)$.
Their allocation-dependent combination is
\begin{align}
\wh W(p,T)
=
\sum_{i=1}^m p_i\wh W_i(T).
\label{eq:surrogate_Gramian}
\end{align}
The surrogate problem uses the same allocation set $\Pset$
as the exact problem.
The associated positive-definite domain is
$\wh{\mathcal X}_T
:=
\left\{
p\in\mathbb R^m
\;\middle|\;
\wh W(p,T)\succ O
\right\}$.
On this domain, the surrogate objective functions are
\begin{align}
\wh f_T(p)
&=
-\log\det\wh W(p,T),
&
\wh g_T(p)
&=
\tr\!\left(\wh W(p,T)^{-1}\right).
\label{eq:surrogate_objectives}
\end{align}
The corresponding optimal-allocation sets are
\begin{align}
\wh{\mathcal S}_{V}(T)
&=
\argmin_{p\in\wh{\mathcal X}_T\cap\Pset}\wh f_T(p),
&
\wh{\mathcal S}_{A}(T)
&=
\argmin_{p\in\wh{\mathcal X}_T\cap\Pset}\wh g_T(p).
\label{eq:surrogate_solutions}
\end{align}
Whenever $\wh{\mathcal X}_T\cap\Pset\neq\emptyset$,
both sets are nonempty by the same argument as in
Section~\ref{subsec3-A}.

\subsection{Relative Gramian Error and Objective Performance}
\label{sec:relative_performance}

We next quantify the performance of surrogate optimal allocations
under the original CAS objectives. The resulting guarantees require
neither uniqueness of the optimizers nor closeness of the allocation
vectors.

For any \(p\in\mathcal X_T\cap\Pset\), define the relative
output-Gramian error
\begin{align}
\delta_{\rm rel}(p)
:=
\left\|
W(p,T)^{-1/2}
\bigl(
\wh W(p,T)-W(p,T)
\bigr)
W(p,T)^{-1/2}
\right\|.
\label{eq:relative_error}
\end{align}
Since the matrix inside the norm is symmetric,
\(\delta_{\rm rel}(p)\leq\delta\) is equivalent to
\begin{align}
(1-\delta)W(p,T)
\preceq
\wh W(p,T)
\preceq
(1+\delta)W(p,T).
\label{eq:sandwich}
\end{align}
Thus, \(\delta_{\rm rel}(p)\) is the smallest
\(\delta\geq0\) for which \eqref{eq:sandwich} holds.
In particular, \(\delta_{\rm rel}(p)<1\) guarantees
\(\wh W(p,T)\succ O\).
 
 The following theorem translates relative
Gramian error bounds into objective-performance guarantees.

\begin{theorem}[Objective guarantees without unique allocations]
\label{thm:objective_guarantees}
Fix any $T>0$.
Suppose that 
$S_{\bullet}(T)\neq \emptyset$ and $\wh S_{\bullet}(T)\neq \emptyset$
for each $\bullet\in\{V,A\}$,
where $S_{\bullet}(T)$ and $\wh S_{\bullet}(T)$
are defined by
\eqref{eq:exact_solutions}
and \eqref{eq:surrogate_solutions}, respectively.
Choose
\begin{align}
p_V^*
&\in \mathcal S_V(T),
&
\wh p_V^*
&\in \wh{\mathcal S}_V(T),
\\
p_A^*
&\in \mathcal S_A(T),
&
\wh p_A^*
&\in \wh{\mathcal S}_A(T),
\end{align}
and suppose that
\(\wh p_V^*,\,\wh p_A^*\in\mathcal X_T\).
Assume further that, for some \(0\leq\delta<1\),
\begin{align}
\delta_{\rm rel}(p_V^*)
&\leq\delta,
&
\delta_{\rm rel}(\wh p_V^*)
&\leq\delta,
\label{eq:relative_error_at_V_optimizers}\\
\delta_{\rm rel}(p_A^*)
&\leq\delta,
&
\delta_{\rm rel}(\wh p_A^*)
&\leq\delta.
\label{eq:relative_error_at_A_optimizers}
\end{align}
Then
\begin{align}
0
\leq
f_T(\wh p_V^*)-f_T(p_V^*)
\leq
r\log\frac{1+\delta}{1-\delta},
\label{eq:Vgap}
\end{align}
and
\begin{align}
1
\leq
\frac{g_T(\wh p_A^*)}{g_T(p_A^*)}
\leq
\frac{1+\delta}{1-\delta}.
\label{eq:Aratio}
\end{align}
\end{theorem}

\begin{proof}
Fix any \(p\in\mathcal X_T\cap\Pset\) satisfying
\(\delta_{\rm rel}(p)\leq\delta<1\), and write
$W:=W(p,T)$ and
$\wh W:=\wh W(p,T)$.
Then \eqref{eq:sandwich} holds and \(\wh W\succ O\).
Taking determinants in \eqref{eq:sandwich} yields
\begin{align}
(1-\delta)^r\det W
\leq
\det\wh W
\leq
(1+\delta)^r\det W,
\end{align}
and hence
\begin{align}
r\log(1-\delta)
\leq
f_T(p)-\wh f_T(p)
\leq
r\log(1+\delta).
\label{eq:relative_V_proof}
\end{align}
Similarly, inversion reverses the Loewner order in
\eqref{eq:sandwich}; taking traces therefore gives
\begin{align}
\frac{1}{1+\delta}g_T(p)
\leq
\wh g_T(p)
\leq
\frac{1}{1-\delta}g_T(p).
\label{eq:relative_A_proof}
\end{align}

We prove \eqref{eq:Vgap}.
The assumption
\(\delta_{\rm rel}(p_V^*)\leq\delta<1\) ensures that
\(p_V^*\in\wh{\mathcal X}_T\cap\Pset\).
 The assumption $\wh p_V^*\in \wh S_V(T)$
 thus gives
$\wh f_T(\wh p_V^*)
\leq
\wh f_T(p_V^*)$.
Combining this inequality with \eqref{eq:relative_V_proof}
at \(p_V^*\) and \(\wh p_V^*\), we obtain
\begin{align}
f_T(\wh p_V^*)
\leq
\wh f_T(p_V^*)+r\log(1+\delta)
\leq
f_T(p_V^*)+r\log\frac{1+\delta}{1-\delta}.
\end{align}
Since \(p_V^*\) is optimal for the exact problem and
\(\wh p_V^*\in\mathcal X_T\cap\Pset\),
we also have \(f_T(p_V^*)\leq f_T(\wh p_V^*)\).
This proves \eqref{eq:Vgap}.

We prove \eqref{eq:Aratio}.
The assumption
\(\delta_{\rm rel}(p_A^*)\leq\delta<1\) ensures that
\(p_A^*\in\wh{\mathcal X}_T\cap\Pset\).
By \eqref{eq:relative_A_proof} and $\wh p_A^*\in \wh S_A(T)$,
\begin{align}
\frac{1}{1+\delta}g_T(\wh p_A^*)
&\leq
\wh g_T(\wh p_A^*)
\nonumber\leq
\wh g_T(p_A^*)
\leq
\frac{1}{1-\delta}g_T(p_A^*).
\end{align}
Since $p_A^*\in S_A(T)$ and
\(\wh p_A^*\in\mathcal X_T\cap\Pset\),
we also have \(g_T(p_A^*)\leq g_T(\wh p_A^*)\).
Dividing by \(g_T(p_A^*)>0\) yields \eqref{eq:Aratio}.
\qed
\end{proof}

Theorem~\ref{thm:objective_guarantees} bounds the loss in the
original CAS objectives. To make these guarantees applicable
to a concrete surrogate model, it remains to control the relative
Gramian error \(\delta_{\rm rel}(p)\) at the relevant allocations.

The relative error is closely related to the absolute discrepancy
between the exact and surrogate output Gramians. Indeed, by
submultiplicativity of the operator norm,
\begin{align}
\delta_{\rm rel}(p)
\leq
\frac{
\|\wh W(p,T)-W(p,T)\|
}{
\lambda_{\min}(W(p,T))
}.
\label{eq:absolute_to_relative}
\end{align}
Thus, a small absolute Gramian discrepancy leads to a small relative
error as long as the exact output Gramian is sufficiently bounded
away from singularity. 

The absolute Gramian discrepancy can, in turn, be controlled through
the discrepancies of the individual candidate Gramians.
In fact, \eqref{eq_W(p,T)} and \eqref{eq:surrogate_Gramian} imply that
$\wh W(p,T)-W(p,T)
=
\sum_{i=1}^m p_i\Delta W_i(T)$,
where \begin{align}
\Delta W_i(T)
:=
\wh W_i(T)-W_i(T)
\quad
(i=1,\ldots,m).
\label{eq:candidate_Gramian_discrepancy}
\end{align}
Therefore, for every \(p\in\Pset\subseteq\Delta_m\),
\begin{align}
\|\wh W(p,T)-W(p,T)\|
&\leq
\max_{1\leq i\leq m}\|\Delta W_i(T)\|.
\label{eq:candidate_to_total_error}
\end{align}
Hence, a uniform bound on the candidate-Gramian discrepancies
\(\|\Delta W_i(T)\|\) yields uniform control of the absolute Gramian
discrepancy over the allocation set and thereby provides a route to
bounding \(\delta_{\rm rel}(p)\). In the next subsection, we derive
such a bound for the TCS surrogate.

\begin{remark}
While \cite{casadei2020model} studies the approximation
of output controllability Gramians and control energy
through model reduction, our analysis quantifies
the performance loss of surrogate-optimal allocations
under the original CAS objectives.
\end{remark}

%%%%%%%%%%%%%%%%%%%%%
\subsection{Approximate Computation of Target VCS/AECS}
\label{sec:target_only_structural}

We now specialize the preceding analysis to the TCS
surrogate. In view of \eqref{eq:candidate_to_total_error}, our main
objective is to derive a uniform bound on the candidate-Gramian
discrepancies \(\|\Delta W_i(T)\|\).

First, we introduce a reduced virtual system that retains only the target-state component.
After reordering the state coordinates so that the target nodes
correspond to the first \(m\) components, 
define
$x(t)=\begin{pmatrix}x_{\rm tar}(t)\\ x_{\rm nt}(t)\end{pmatrix}$,
where $x_{\rm tar}(t)\in\mathbb{R}^m$ collects the target states and $x_{\rm nt}(t)\in\mathbb{R}^{n-m}$ collects the non-target states.
Accordingly, we partition the system matrix $A$ as
\begin{align}
  A = \begin{pmatrix} 
  A_{11} & A_{12} \\ 
  A_{21} & A_{22} 
  \end{pmatrix},
  \qquad
  A_{11}=CAC^\top\in {\bb R}^{m\times m}, \label{eq_A_partition}
\end{align}
where
\begin{align}
    C &:=\begin{pmatrix}
  I_m & 0  
\end{pmatrix}
= \sum_{i=1}^m e_i^{(m)}\left(e_i^{(n)}\right)^\top. \label{def_C}
\end{align}
Neglecting the influence of the non-target component $x_{\rm nt}(t)$ on the target dynamics, we obtain the reduced virtual system
\begin{align}
    \dot{x}_{\rm tar}(t) = A_{11} x_{\rm tar}(t) + {\rm diag}(\sqrt{p_1},\ldots, \sqrt{p_m}) u(t),
 \label{eq:lti_virtual4}
\end{align}
which evolves only on the target subspace.

To obtain an upper bound for $\|\Delta W_i(T)\|$, we first establish an integral representation of the difference
\begin{align}
  X(t):=\exp(A_{11}t)-C\exp(At)C^\top \label{eq_X}
\end{align}
using a variation-of-constants argument. For this purpose, define
\begin{align}
    E:=\begin{pmatrix}
        0 & -A_{12}
    \end{pmatrix}. \label{def_E}
\end{align}

\begin{lemma} \label{Lem_Key_diffW}
With \(A_{11}\), \(C\), \(X(t)\), and \(E\) defined in \eqref{eq_A_partition}, \eqref{def_C}, \eqref{eq_X}, and \eqref{def_E}, respectively, we have
    \begin{align}
       X(t) = \int_0^t \exp(A_{11}(t-s)) E \exp(A s) C^\top \D s. \label{key_relation}
    \end{align}
\end{lemma}
\begin{proof}
    Let $Y(t):=C\exp (At)$. Then,
    \begin{align}
        \dot{Y}(t) = CA\exp(At) = A_{11}Y(t)-E\exp(At) \label{diffeq_X}
    \end{align}
    with $Y(0)=C$.
    Here, the second equality follows from
    \begin{align*}
        CA = \begin{pmatrix}
            A_{11} & A_{12}
        \end{pmatrix} = A_{11}C -E.
    \end{align*}
    The solution to \eqref{diffeq_X} is given by
    \begin{align}
        Y(t) = \exp(A_{11}t)Y(0) -\int_0^t \exp(A_{11}(t-s))E \exp(As) \D s. \label{solution_X}
    \end{align}
    Substituting \eqref{solution_X} into the right hand side of
    \begin{align*}
        X(t) = (\exp(A_{11}t)C-Y(t))C^\top,
    \end{align*}
    we obtain \eqref{key_relation}. \qed
\end{proof}

To obtain an upper bound for $\|\Delta W_i(T)\|$, we next derive exponential bounds on the matrix exponentials $\exp(At)$ and $\exp(A_{11}t)$, based on the logarithmic norm \cite{trefethen2005spectra} defined by 
\begin{align*}
\mu(A):=\lambda_{\max}\left(\frac{A+A^\top}{2}\right).
\end{align*}

\begin{lemma} \label{lem_Gronwall}
Under the partition \eqref{eq_A_partition},
 for all $t\geq 0$,
\begin{align*}
        \|\exp(At)\| \leq \e^{\mu(A) t},\quad \|\exp(A_{11}t)\| \leq \e^{\mu(A_{11})t}\leq  \e^{\mu(A) t}.
\end{align*}
\end{lemma}
\begin{proof}
As described in \cite[Chapter~IV]{trefethen2005spectra}, for all $t\geq 0$,
\begin{align*}
\|\exp(At)\| \leq \e^{\mu(A) t},\qquad 
\|\exp(A_{11}t)\| \leq \e^{\mu(A_{11}) t}.
\end{align*}
Since $\mu(A_{11})\leq \mu(A)$, the desired inequality follows.
In fact,
since $(A_{11}+A_{11}^\top)/2$ is a principal submatrix of $(A+A^\top)/2$, Cauchy's interlacing theorem (see \cite[Theorem~4.3.28]{horn2012matrix}) yields
\begin{align*}
\mu(A_{11})=\lambda_{\max}\left(\frac{A_{11}+A_{11}^\top}{2}\right) \leq \lambda_{\max}\left(\frac{A+A^\top}{2}\right) = \mu(A).
\end{align*}
This completes the proof. \qed
\end{proof}

If the logarithmic norm satisfies $\mu(A)<0$, then both $A$ and $A_{11}$ are stable, since Lemma~\ref{lem_Gronwall} implies that
\begin{align*}
\|\exp(At)\|\le \e^{\mu(A)t}\to 0,\quad    \|\exp(A_{11}t)\|\le \e^{\mu(A)t}\to 0
\end{align*}
  as $t\to\infty$. However, stability does not, in general, transfer between a matrix and its principal submatrix: for example, $A=\begin{pmatrix}1&1\\-3&-2\end{pmatrix}$ is stable whereas $A_{11}=1$ is unstable, while $A=\begin{pmatrix}-1&10\\ 0&1\end{pmatrix}$ is unstable even though $A_{11}=-1$ is stable. These examples demonstrate that relying solely on the stability of either block can be misleading. In contrast, verifying $\mu(A)<0$ guarantees stability of both the full and reduced systems, highlighting the importance of analyzing $\mu(A)$ as a unified stability criterion.

Using Lemmas \ref{Lem_Key_diffW} and \ref{lem_Gronwall},
we obtain the following expression and bound on $\Delta W_i(T)$.

\begin{theorem} \label{thm:explicit-DeltaWi}
Let \(C\) and \(E\) be defined by \eqref{def_C} and \eqref{def_E}, and let \(A\) be partitioned as in \eqref{eq_A_partition}. For any \(T>0\), define
\(\Delta W_i(T)\) as in \eqref{eq:candidate_Gramian_discrepancy}.
Then,
    \begin{align}
        \Delta W_i(T) = \int_0^T G_i(t) \D t, \label{Expression_Delta_Wi}
    \end{align}
    where
    \begin{align}
    &G_i(t) \label{Expression_Gi}\\
    &:=\int_0^t
    \exp(A_{11}(t-s)) E \exp(As)
    e_i^{(n)}(e_i^{(m)})^\top
    \exp(A_{11}^\top t) \D s \nonumber\\
    &+\int_0^t
    C\exp(At)e_i^{(n)}
    (e_i^{(n)})^\top
    \exp(A^\top s) E^\top
    \exp(A_{11}^\top (t-s)) \D s.
\end{align}
Moreover,
\begin{align}
\|\Delta W_i(T)\|\le \Phi_{\mu(A)}(T)\|A_{12}\|, \label{norm_Delta_Wi}
\end{align}
where
\begin{align}
\Phi_{\mu(A)}(T):=
\begin{cases}
    \frac{\e^{2\mu(A) T}(2\mu(A) T-1)+1}{2\mu(A)^2}\quad (\mu(A)\neq 0)\\
     T^2\quad (\mu(A)=0).
\end{cases} \label{eq_Phi_constant}
\end{align}
In particular,
    if $\mu(A)<0$, then $\lim_{T\rightarrow \infty} \Delta W_i(T)$ exists and
    \begin{align*}
        \lim_{T\rightarrow \infty} \|\Delta W_i(T)\| \leq \frac{1}{2\mu(A)^2} \|A_{12}\|.
    \end{align*}
\end{theorem}
\begin{proof}
First, we establish \eqref{Expression_Delta_Wi}.
Definitions \eqref{eq:candidate_Gramian_discrepancy} and \eqref{eq_X} imply that
    \begin{align}
        \Delta W_i(T)&= \int_0^T\left( \exp(A_{11}t)e_i^{(m)}(e_i^{(m)})^\top \exp(A_{11}^\top t) \right.\\
        &\quad \left. -C\exp(At)e_i^{(n)}(e_i^{(n)})^\top \exp(A^\top t)C^\top \right) \D t \\
        &=\int_0^T \left(X(t)e_i^{(m)}(e_i^{(m)})^\top \exp(A_{11}^\top t) \right.\\
        &\quad \left.+ C\exp(At)C^\top e_i^{(m)}(e_i^{(m)})^\top X(t)^\top \right) \D t, \label{Cal_Delta_W}
    \end{align}
    where
    the second equality follows from the identity
       $C^\top e_i^{(m)} = e_i^{(n)}$. 
    Applying Lemma \ref{Lem_Key_diffW} to \eqref{Cal_Delta_W} yields \eqref{Expression_Delta_Wi}.

Next, we derive \eqref{norm_Delta_Wi}.
    Using \eqref{Expression_Gi}, sub-multiplicativity of the operator norm $\|\cdot\|$,
$\|e_i^{(n)}(e_i^{(m)})^\top\|=1$, $\|E\|=\|A_{12}\|$, and Lemma~\ref{lem_Gronwall}, the first integrand in \eqref{Expression_Gi} can be bounded as
\begin{align}
&\bigl\|
   \exp(A_{11}(t-s))
   E
   \exp(A s)
   e_i^{(n)}(e_i^{(m)})^\top
   \exp(A_{11}^{\top}t)
\bigr\| \\
 &\leq
   \|\exp(A_{11}(t-s))\|\cdot
   \|E\|\cdot
   \|\exp(A s)\|\cdot
   \|e_i^{(n)}(e_i^{(m)})^\top\| \\
   &\quad \cdot \|\exp(A_{11}^\top t)\| \\
 &\leq
    \e^{\mu(A)(t-s)}
   \|A_{12}\|
    \e^{\mu(A) s}
    \e^{\mu(A) t} = \|A_{12}\| \e^{2\mu(A) t}.
\end{align}
The second integrand in \eqref{Expression_Gi} is bounded in the same
way, and hence
\begin{align*}    
   \|G_i(t)\|
      \le
      2
      \int_{0}^{t}
         \|A_{12}\|\e^{2\mu(A) t}
      \D s
      =
      2\|A_{12}\|
      t\e^{2\mu(A) t}.
      \end{align*}
Therefore,
\begin{align*}
    \|\Delta W_i(T)\| \leq \int_0^T \|G_i(t)\| \D t =  2\|A_{12}\|\int_{0}^{T} t\,\e^{2\mu(A) t} \D t,
\end{align*}
which yields
\eqref{norm_Delta_Wi}. \qed
\end{proof}

The bound in \eqref{norm_Delta_Wi} shows that
the approximation error scales linearly with the cross-coupling magnitude
$\|A_{12}\|$, with a time- and dynamics-dependent prefactor $\Phi_{\mu(A)}(T)$.
In block decomposition \eqref{eq_A_partition},
the block $A_{12}$ encodes the direct influence from the lower-level
subsystem to the higher-level subsystem.
Its norm therefore quantifies the inter-layer coupling: a small $\|A_{12}\|$
implies weak upward influence, so the reduced Gramian $\wh W(p,T)$
closely approximates the full output-controllability Gramian $W(p,T)$.
Theorem~\ref{thm:explicit-DeltaWi} shows that
the influence of non-target dynamics can be
small when
$\Phi_{\mu(A)}(T)\|A_{12}\|$ is sufficiently small.
This expectation is numerically supported in Section~\ref{Sec_numerical}, where the results indicate that the bound provides a useful estimate for small horizons $T$, and that networks with smaller~$\|A_{12}\|$ tend to exhibit a smaller approximation error in the TCS.

%%%%%%%%%
\section{Numerical experiments using real-world human brain network data} \label{Sec_numerical}

We evaluate target VCS/AECS, along with their reduced approximations, using a human brain-network dataset from \cite{vskoch2022human}. 
The dataset contains structural connectivity matrices for $88$ individuals. 
Each individual's brain network is represented by a $90\times 90$ matrix, 
where the $(i,j)$ entry is the proportion of tractography streamlines seeded in the \(i\)th region of interest (ROI) that reach the \(j\)th ROI, based on the Automatic Anatomical Labeling atlas. 
Thus, the dataset comprises brain networks for $88$ individuals, 
each consisting of $90$ nodes corresponding to distinct brain regions. 
We note that this is the same dataset used in \cite[Section IV]{sato2025uniqueness}.
Our analysis uses only the publicly available anonymized connectivity matrices released in \cite{vskoch2022human}.

We model the individual blood oxygen level-dependent (BOLD) signal dynamics as 
the continuous-time linear system
\begin{align}
   \dot{x}^{(i)}(t) = A^{(i)}x^{(i)}(t) \quad (i=1,\ldots,88), \label{ex_continuous}
\end{align}
where each component of $x^{(i)}(t)$ denotes the BOLD signal of each ROI at time $t$ for the $i$th individual.
The system matrix $A^{(i)}$ is defined by
$A^{(i)} := -\mathcal{L}^{(i)}$,
where $\mathcal{L}^{(i)}$ is the graph Laplacian
\begin{align*}
   \mathcal{L}^{(i)} := 
   {\rm diag}\!\left(\sum_{j=1}^{90}\mathcal{C}^{(i)}_{1j},\,\ldots,\,\sum_{j=1}^{90}\mathcal{C}^{(i)}_{90j}\right) 
   - \mathcal{C}^{(i)}.
\end{align*}
Here $\mathcal{C}^{(i)}\in\mathbb{R}^{90\times 90}$ denotes the transpose of the structural connectivity matrix 
for the $i$th individual. 
The matrices $\mathcal{C}^{(i)}$ are nonsymmetric, although they are nearly symmetric, as mentioned in \cite{vskoch2022human}.

Based on system model \eqref{ex_continuous}, 
we construct the target selection matrices for both VCS and AECS as follows.
For each individual $i\in\{1,\ldots,88\}$ and for a given time horizon $T>0$, 
let $\{s^{\mathrm{VCS}}_1(T),\ldots,s^{\mathrm{VCS}}_m(T)\}$ 
and $\{s^{\mathrm{AECS}}_1(T),\ldots,s^{\mathrm{AECS}}_m(T)\}$ 
denote the indices of the top $m$ ROIs 
ranked by the average VCS and AECS values at horizon $T$, respectively.
Without loss of generality, we relabel the coordinates so that 
the selected ROIs appear in the first $m$ positions of the state vector.
For each $i$, this yields a block partition of the system matrix
\begin{align*}
   A^{(i)}_{\bullet} =
   \begin{pmatrix}
     A_{11,\bullet}^{(i)} & A_{12,\bullet}^{(i)} \\
     A_{21,\bullet}^{(i)} & A_{22,\bullet}^{(i)}
   \end{pmatrix}, \quad
   A_{11,\bullet}^{(i)}\in\mathbb{R}^{m\times m},
\end{align*}
where $\bullet\in\{\mathrm{VCS},\mathrm{AECS}\}$ indicates whether 
the target set is determined by VCS or AECS. 
Here, $A_{11,\bullet}^{(i)}$ represents the dynamics restricted to the selected target set, 
whereas $A_{12,\bullet}^{(i)}$ captures the coupling from non-target nodes into the targets. 

For each individual $i\in \{1,\ldots, 88\}$, let $(p^{\mathrm{VCS}})^{(i)}$ and $(p^{\mathrm{AECS}})^{(i)}$ 
denote the target VCS and target AECS, respectively, 
and let $(p^{\mathrm{VCS}}_{\mathrm{red}})^{(i)}$ and $(p^{\mathrm{AECS}}_{\mathrm{red}})^{(i)}$ 
denote their reduced-system counterparts.
Corollary~\ref{cor:generic_uniqueness_TCS} in this paper and
\cite[Theorem~1]{sato2025uniqueness} guarantee that these scores are unique for almost all $T>0$.

\begin{remark}
Standard models of brain activity include nonlinearities in the mapping from neural activity to the hemodynamic/BOLD response \cite{friston2000nonlinear,zeidman2019guide}. In this work, however, we adopt a linear Laplacian-based dynamics described in \eqref{ex_continuous} as a first-order approximation near the resting state. A previous study has shown that, for resting-state low-frequency correlations, linear diffusion on the structural network can capture the principal second-order statistics and, in some cases, match or even outperform more complex nonlinear neural-mass models in predicting functional connectivity from structural connectivity \cite{abdelnour2014network}. This approach has also been supported by prior work in the literature \cite{gu2015controllability, tu2018warnings, pasqualetti2019re, suweis2019brain}.
\end{remark}

%%%%%%%%%%%%%%%%
\subsection{Comparison of Target and Non-Target Coupling for VCS and AECS} \label{Subsec_comparison}

To assess the influence of non-target regions on the target dynamics, 
we evaluated the operator norms $\|A_{12,\bullet}^{(i)}\|$ across all individuals. 
For each selection method $\bullet\in\{\mathrm{VCS},\mathrm{AECS}\}$, 
define $\{\|A_{12,\bullet}^{(i)}\|\}_{i=1}^{88}$ as the collection of values obtained 
across all subjects. 
We summarize these results using the sample mean 
$\overline{\|A_{12,\bullet}\|}:=\frac{1}{88}\sum_{i=1}^{88}\|A_{12,\bullet}^{(i)}\|$
and the corresponding population standard deviation $\sqrt{\frac{1}{88}\sum_{i=1}^{88}
   \left(\|A_{12,\bullet}^{(i)}\|-\overline{\|A_{12,\bullet}\|}\right)^2 }$.

Table~\ref{tab:summary_all} reveals a consistent pattern across all examined
configurations $(T,m)$: the cross-coupling magnitudes $\|A_{12,\bullet}\|$
associated with VCS are uniformly smaller than those associated with AECS,
irrespective of the choice of time horizon $T$ or the number of selected
targets $m$. This suggests that VCS-based target selection tends to identify
subsets of regions whose dynamics are less influenced by non-target nodes,
resulting in weaker interference from the complementary block compared to
AECS-based selection.

It should be noted, however, that this advantage in terms of cross-coupling
does not necessarily imply better approximation accuracy of the target VCS/AECS, as shown in the next
subsection.

\begin{table*}[t]
  \centering
  \caption{Cross-coupling magnitudes
(mean $\pm$ standard deviation) across 88 subjects 
for different $(T,m)$ configurations.}
  \label{tab:summary_all}
  \begin{tabular}{cccll}
    \toprule
    $T$ & $m$ & $\|A_{12, {\rm VCS}}\|$ (mean$\pm$std) & $\|A_{12, {\rm AECS}}\|$ (mean$\pm$std) \\
    \midrule
     1   & 3   & $3.124{\times}10^{-1}\,\pm\,8.392{\times}10^{-2}$ & 
                 $7.364{\times}10^{-1}\,\pm\,8.055{\times}10^{-2}$ \\
     1   & 10  & $2.568{\times}10^{-1}\,\pm\,3.512{\times}10^{-2}$ & 
                 $8.004{\times}10^{-1}\,\pm\,8.385{\times}10^{-2}$ \\
     1   & 30  & $3.939{\times}10^{-1}\,\pm\,4.134{\times}10^{-2}$ & 
                 $8.123{\times}10^{-1}\,\pm\,8.745{\times}10^{-2}$ \\
   100   & 3   & $1.218{\times}10^{-1}\,\pm\,3.743{\times}10^{-2}$ & 
                 $7.420{\times}10^{-1}\,\pm\,8.076{\times}10^{-2}$ \\
   100   & 10  & $2.409{\times}10^{-1}\,\pm\,4.249{\times}10^{-2}$ & 
                 $8.004{\times}10^{-1}\,\pm\,8.385{\times}10^{-2}$ \\
   100   & 30  & $3.828{\times}10^{-1}\,\pm\,4.343{\times}10^{-2}$ & 
                 $8.120{\times}10^{-1}\,\pm\,8.564{\times}10^{-2}$ \\
    \bottomrule
  \end{tabular}
\end{table*}

%%%%%%%%%%%%%%%%%%%
\subsection{Approximation Error Between Target and Reduced Formulations}
Let $\{\|{\rm diff}_{\bullet}^{(i)}\|\}_{i=1}^{88}$ denote the norm difference between the TCS and the reduced-system CS for $\bullet\in\{\mathrm{VCS},\mathrm{AECS}\}$. 
That is, for each subject $i$, 
    $\|{\rm diff}_{\bullet}^{(i)}\| = 
      \|(p^{\bullet})^{(i)} - (p^{\bullet}_{\mathrm{red}})^{(i)}\|$.
We report the sample mean
$\overline{\|{\rm diff}_{\bullet}\|} := \frac{1}{88}\sum_{i=1}^{88}\|{\rm diff}_{\bullet}^{(i)}\|$
and the corresponding population standard deviation $\sqrt{\frac{1}{88}\sum_{i=1}^{88}\left(\|{\rm diff}_{\bullet}^{(i)}\|-\overline{\|{\rm diff}_{\bullet}\|}\right)^2 }$.

Table~\ref{tab:diffnorm_summary} summarizes the norm differences between the target and reduced formulations of VCS and AECS across 88 subjects under various $(T,m)$ configurations. At $T=1$, VCS exhibits substantially smaller differences than AECS for $m=3, 10, 30$, indicating that the reduced formulation closely approximates the target VCS in the short-horizon setting. This observation is qualitatively consistent with the short-horizon Gramian error bound in Theorem~\ref{thm:explicit-DeltaWi} since $\Phi_{\mu(A)}(T)=T^2+O(T^3)$ as $T\downarrow0$.

However, as $T$ increases, the situation changes markedly. When $T=100$, VCS produces significantly larger discrepancies than AECS regardless of the value of $m$. For the 88 subjects, $\mu(A)$ is positive and ranges from $3.81\times10^{-2}$ to $8.37\times10^{-2}$ (mean $5.22\times10^{-2}$). Hence, the factor $\Phi_{\mu(A)}(T)$ in Theorem~\ref{thm:explicit-DeltaWi} grows rapidly with $T$, making the resulting error bound uninformative at long horizons. This bound does not, however, explain the difference in approximation accuracy between VCS and AECS. Although VCS yields smaller $\|A_{12}\|$, this does not translate into smaller allocation errors at long horizons: VCS, despite being favorable at $T=1$, performs worse than AECS at $T=100$. By contrast, AECS achieves relatively consistent accuracy across both short and long horizons, irrespective of $m$, thereby demonstrating its robustness in approximating target AECS across varying time scales.

\begin{table*}[t]
  \centering
  \caption{Differences between target and reduced formulations 
  (mean $\pm$ std of $\|{\rm diff}_{\bullet}\|$) over 88 subjects 
  for various $(T,m)$ configurations.}
  \label{tab:diffnorm_summary}
  \begin{tabular}{cccll}
    \toprule
    $T$ & $m$ & $\|{\rm diff}_{\rm VCS}\|$ (mean$\pm$std) & $\|{\rm diff}_{\rm AECS}\|$ (mean$\pm$std) \\
    \midrule
     1   & 3   & $0.000{\times}10^{0}\,\pm\,0.000{\times}10^{0}$ & 
                  $1.827{\times}10^{-3}\,\pm\,7.529{\times}10^{-4}$ \\
     1   & 10  & $4.465{\times}10^{-5}\,\pm\,4.940{\times}10^{-5}$ & 
                  $1.132{\times}10^{-3}\,\pm\,2.391{\times}10^{-4}$ \\
     1   & 30  & $6.471{\times}10^{-5}\,\pm\,3.177{\times}10^{-5}$ & 
                  $6.887{\times}10^{-4}\,\pm\,8.818{\times}10^{-5}$ \\
   100   & 3   & $2.281{\times}10^{-2}\,\pm\,2.142{\times}10^{-2}$ & 
                  $4.177{\times}10^{-3}\,\pm\,2.046{\times}10^{-3}$ \\
   100   & 10  & $4.082{\times}10^{-2}\,\pm\,1.841{\times}10^{-2}$ & 
                  $3.252{\times}10^{-3}\,\pm\,9.036{\times}10^{-4}$ \\
   100   & 30  & $1.992{\times}10^{-2}\,\pm\,5.357{\times}10^{-3}$ & 
                  $4.399{\times}10^{-3}\,\pm\,5.215{\times}10^{-4}$ \\
    \bottomrule
  \end{tabular}
\end{table*}

Figures~\ref{fig:AECS_T1_m30}--\ref{fig:VCS_T100_m30} collectively show that the reduced-system approximation is faithful for target VCS/AECS at the short horizon $(T=1)$, but only AECS remains well approximated at the long horizon $(T=100)$, whereas VCS exhibits sizable discrepancies between the target and reduced formulations. This qualitative pattern aligns with the node-level evidence in Table~\ref{tab:metrics_nodes_checked}: Target AECS (TAECS) selects exactly the same top 5 regions at $T=1$ and $T=100$, indicating temporal stability of the target set, while target VCS (TVCS) selects markedly different regions across horizons (e.g., middle orbital gyrus, cingulum, cuneus, pallidum at $T=1$ versus amygdala and Heschl’s gyrus at $T=100$), revealing strong sensitivity to $T$. The mismatch is particularly evident at $(T,m)=(100,30)$, where the reduced VCS elevates Node~36 (Right Cingulum Post) and Node~69 (Left Paracentral Lobule) into its top 5, even though neither appears in the corresponding TVCS results---clear evidence that, at longer horizons, the reduced VCS fails to track its target formulation, whereas AECS retains robust agreement across time scales.

\begin{figure}[t]
   \centering
   \begin{minipage}[b]{0.49\linewidth}
      \centering
      \includegraphics[width=\linewidth]{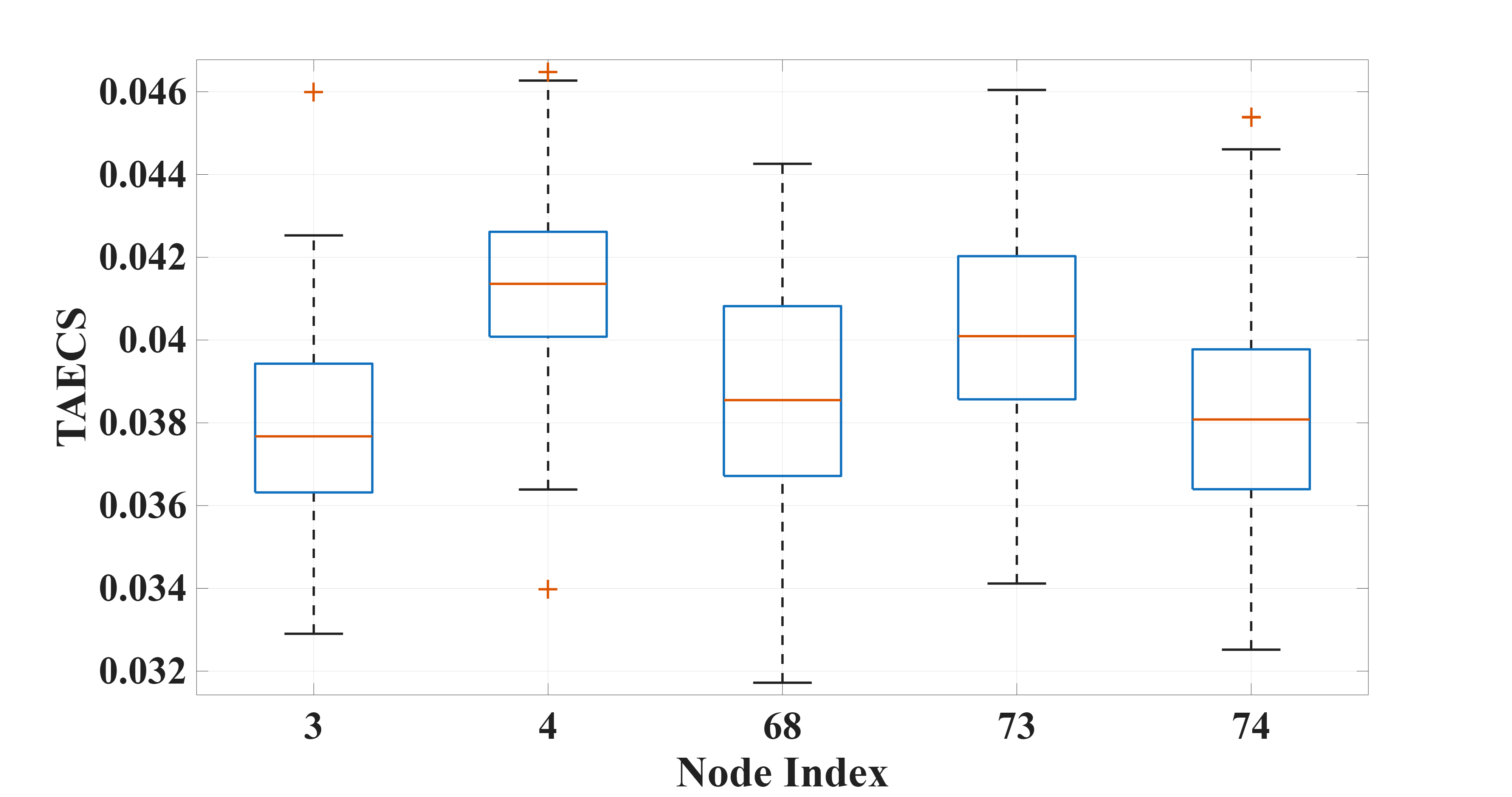}
   \end{minipage}
   \hfill
   \begin{minipage}[b]{0.49\linewidth}
      \centering
      \includegraphics[width=\linewidth]{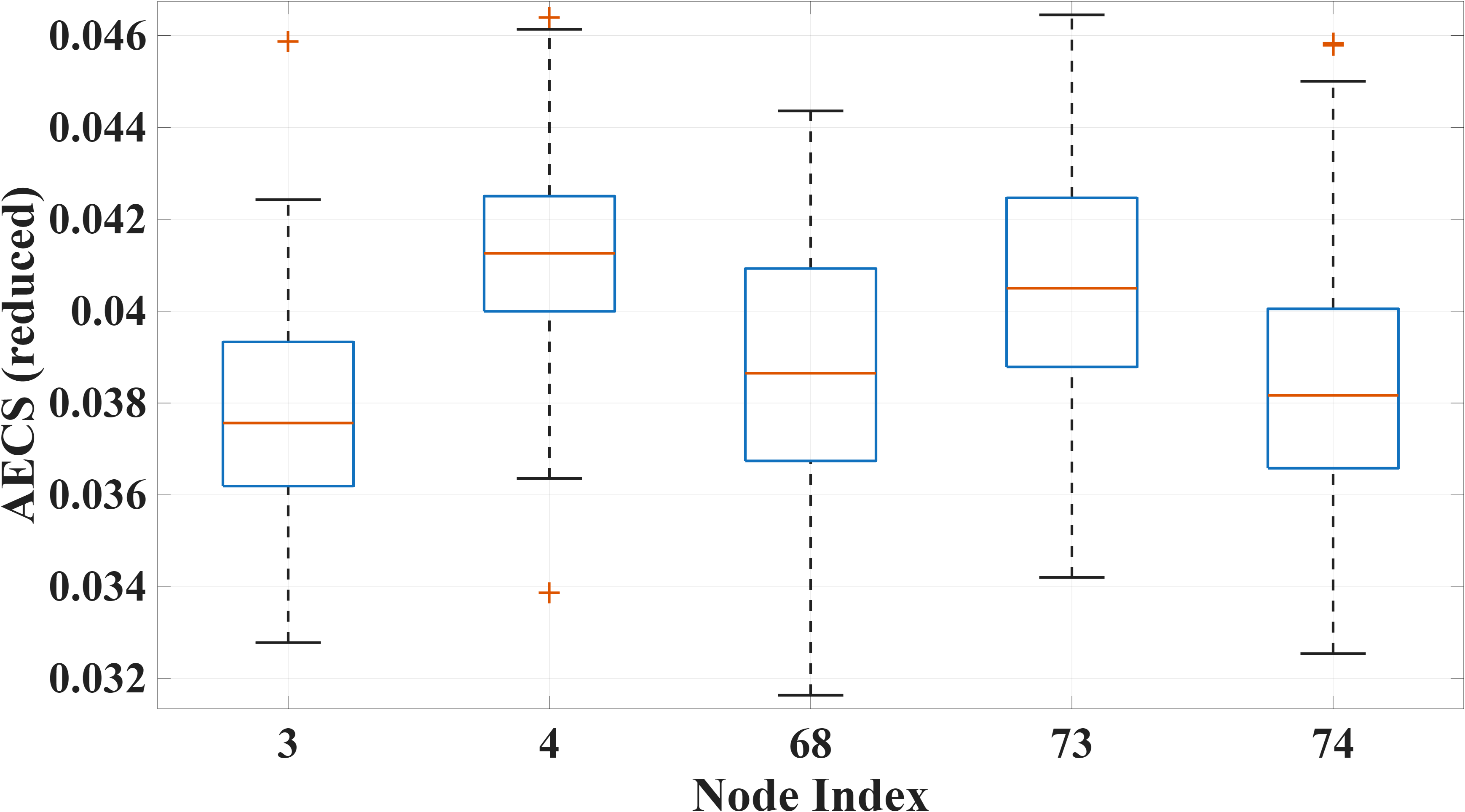}
   \end{minipage}
   \caption{Boxplots of the top 5 nodes for $(T,m)=(1,30)$: target AECS (left) and its reduced-system approximation (right).}
   \label{fig:AECS_T1_m30}
\end{figure}

\begin{figure}[t]
   \centering
   \begin{minipage}[b]{0.49\linewidth}
      \centering
      \includegraphics[width=\linewidth]{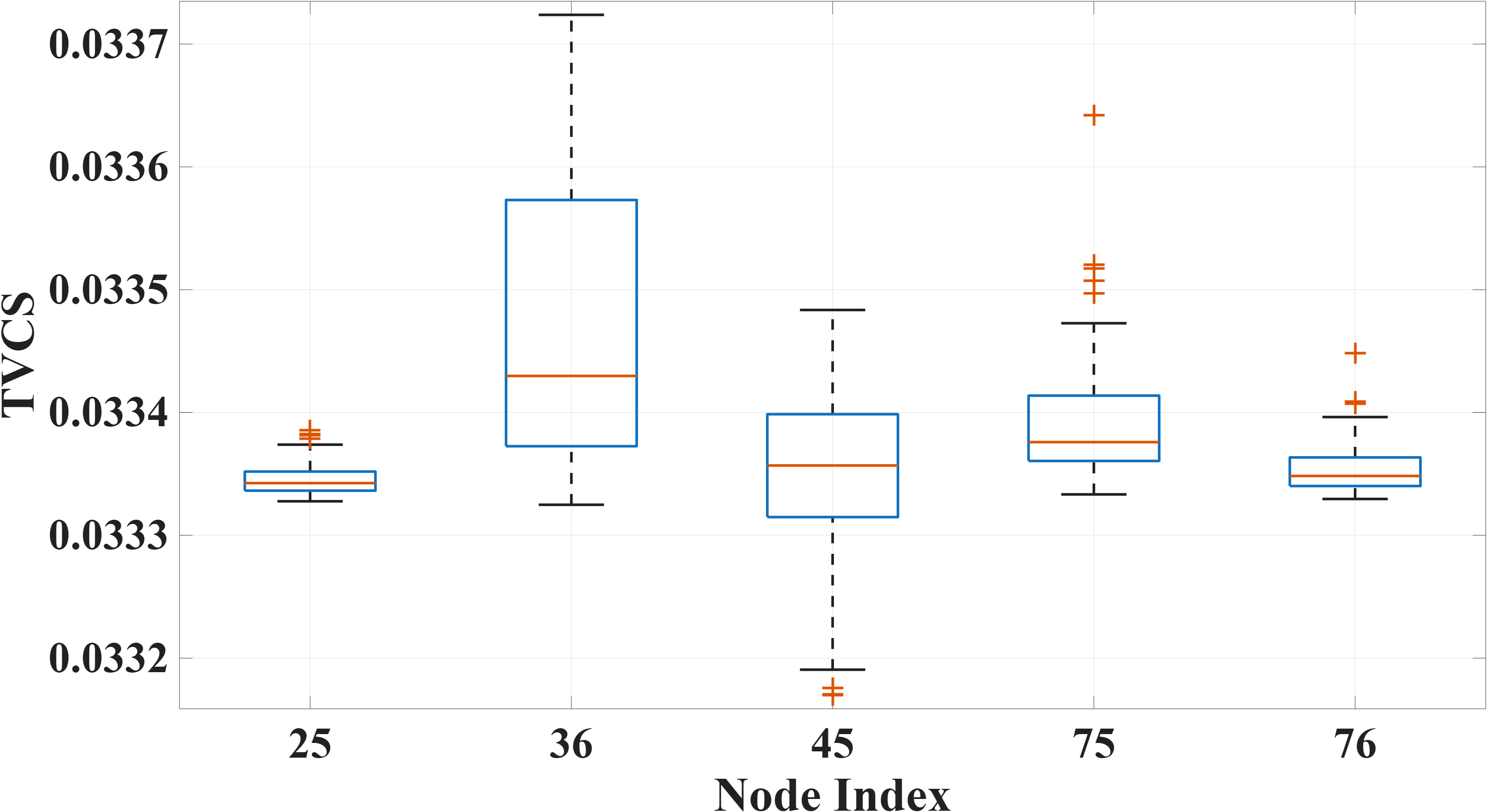}
   \end{minipage}
   \hfill
   \begin{minipage}[b]{0.49\linewidth}
      \centering
      \includegraphics[width=\linewidth]{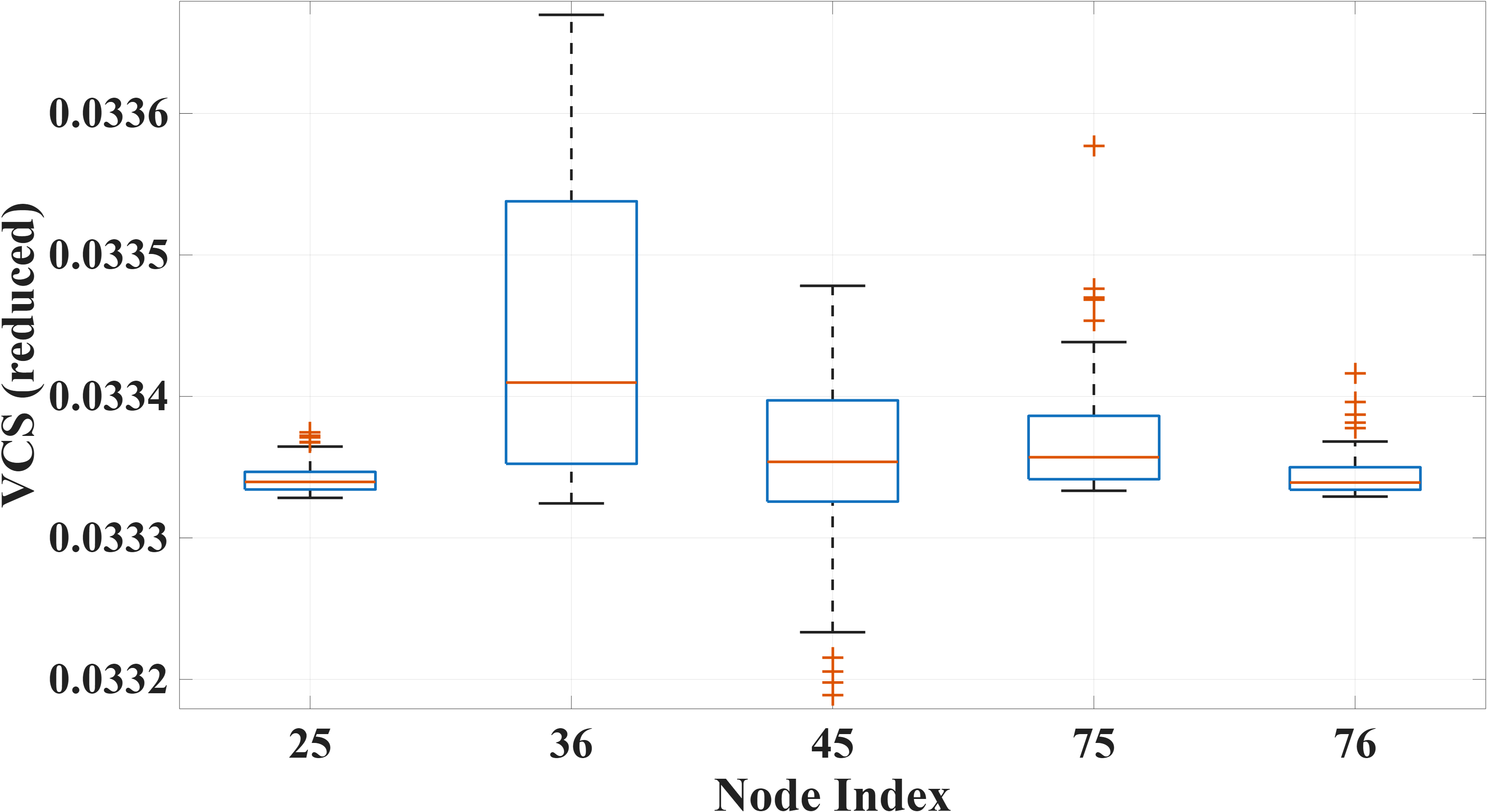}
   \end{minipage}
   \caption{Boxplots of the top 5 nodes for $(T,m)=(1,30)$: target VCS (left) and its reduced-system approximation (right).}
   \label{fig:VCS_T1_m30}
\end{figure}

\begin{comment}
    
\begin{figure}[t]
   \centering
   \includegraphics[width=1.05\linewidth]{TAECS_T1_m30.png}\\[1em]
   \includegraphics[width=0.9\linewidth]{AECS_red_T1_m30.png}
   \caption{Boxplots of the top 5 nodes: target AECS (top) and its reduced-system approximation (bottom) for $(T,m)=(1,30)$.}
   \label{fig:AECS_T1_m30}
\end{figure}

\begin{figure}[t]
   \centering
   \includegraphics[width=0.9\linewidth]{TVCS_T1_m30.png}\\[1em]
   \includegraphics[width=0.9\linewidth]{VCS_red_T1_m30.png}
   \caption{Boxplots of the top 5 nodes: target VCS (top) and its reduced-system approximation (bottom) for $(T,m)=(1,30)$.}
   \label{fig:VCS_T1_m30}
\end{figure}

\begin{figure}[t]
   \centering
   \includegraphics[width=1\linewidth]{TAECS_T100_m30.png}\\[1em]
   \includegraphics[width=0.85\linewidth]{AECS_red_T100_m30.png}
   \caption{Boxplots of the top 5 nodes: target AECS (top) and its reduced-system approximation (bottom) for $(T,m)=(100,30)$.}
   \label{fig:AECS_T100_m30}
\end{figure}

\begin{figure}[t]
   \centering
   \includegraphics[width=0.9\linewidth]{TVCS_T100_m30.png}\\[1em]
   \includegraphics[width=0.9\linewidth]{VCS_red_T100_m30.png}
   \caption{Boxplots of the top 5 nodes: target VCS (top) and its reduced-system approximation (bottom) for $(T,m)=(100,30)$.}
   \label{fig:VCS_T100_m30}
\end{figure}

\end{comment}

\begin{figure}[t]
   \centering
   \begin{minipage}[b]{0.49\linewidth}
      \centering
      \includegraphics[width=\linewidth]{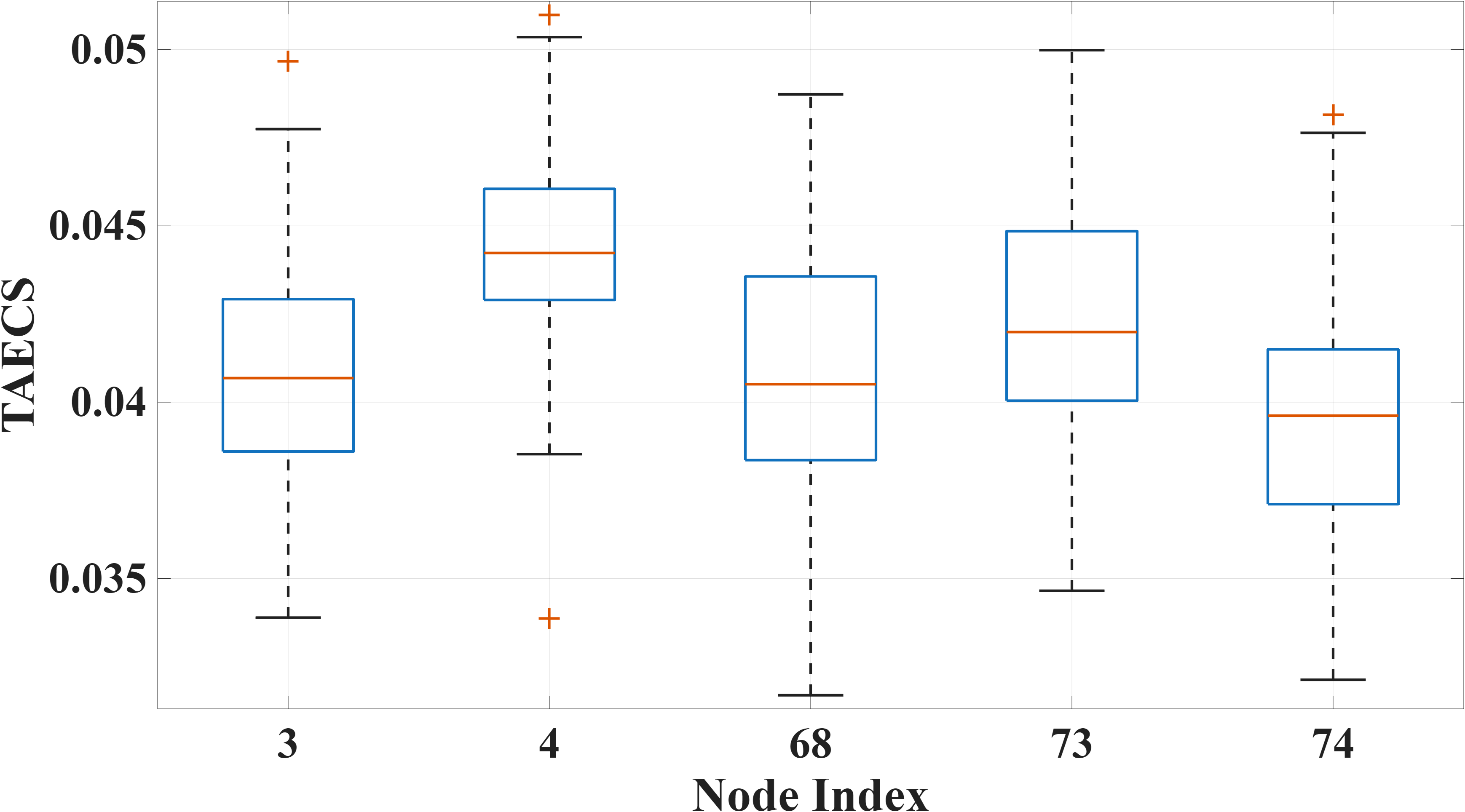}
   \end{minipage}
   \hfill
   \begin{minipage}[b]{0.49\linewidth}
      \centering
      \includegraphics[width=\linewidth]{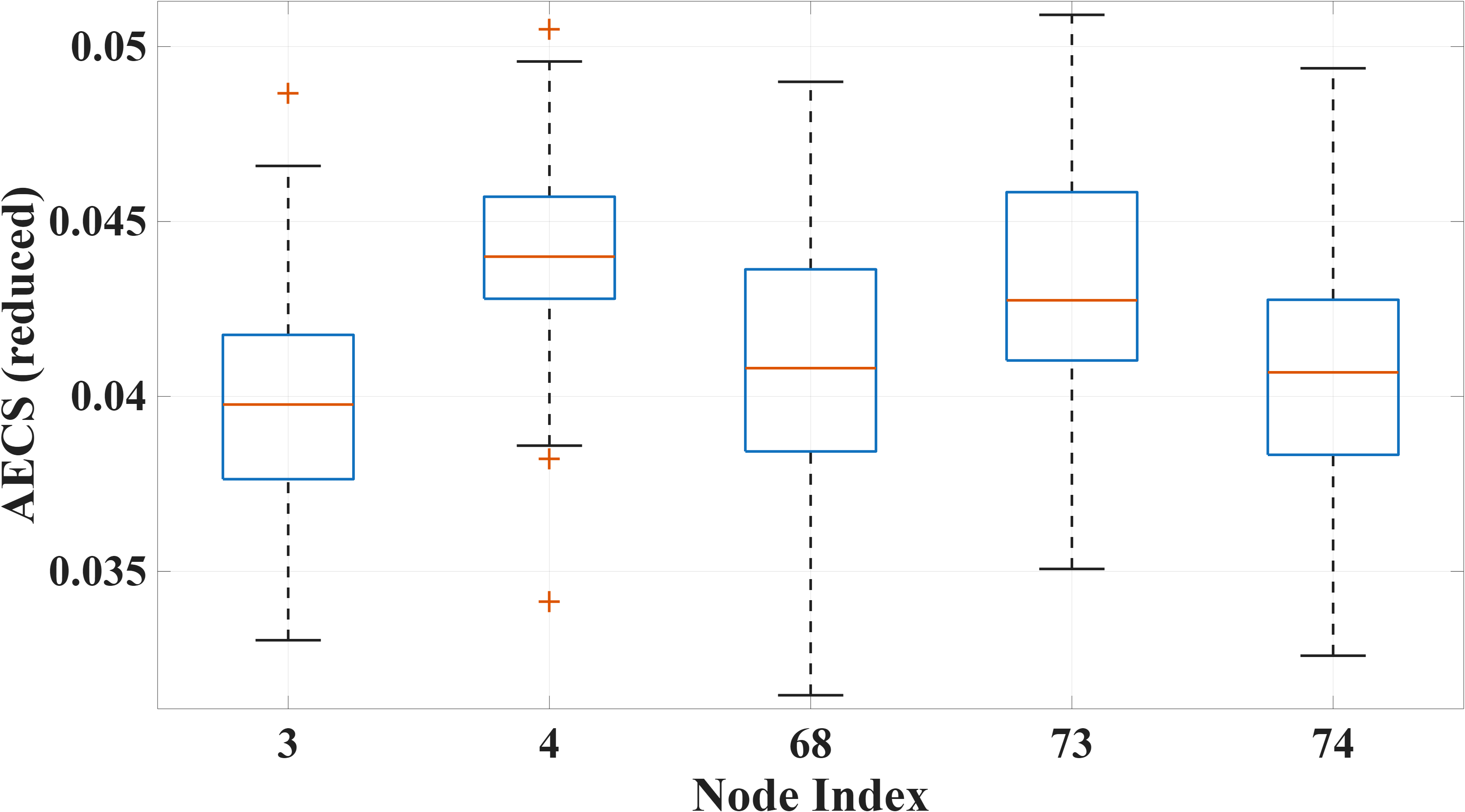}
   \end{minipage}
   \caption{Boxplots of the top 5 nodes for $(T,m)=(100,30)$: target AECS (left) and its reduced-system approximation (right).}
   \label{fig:AECS_T100_m30}
\end{figure}

\begin{figure}[t]
   \centering
   \begin{minipage}[b]{0.49\linewidth}
      \centering
      \includegraphics[width=\linewidth]{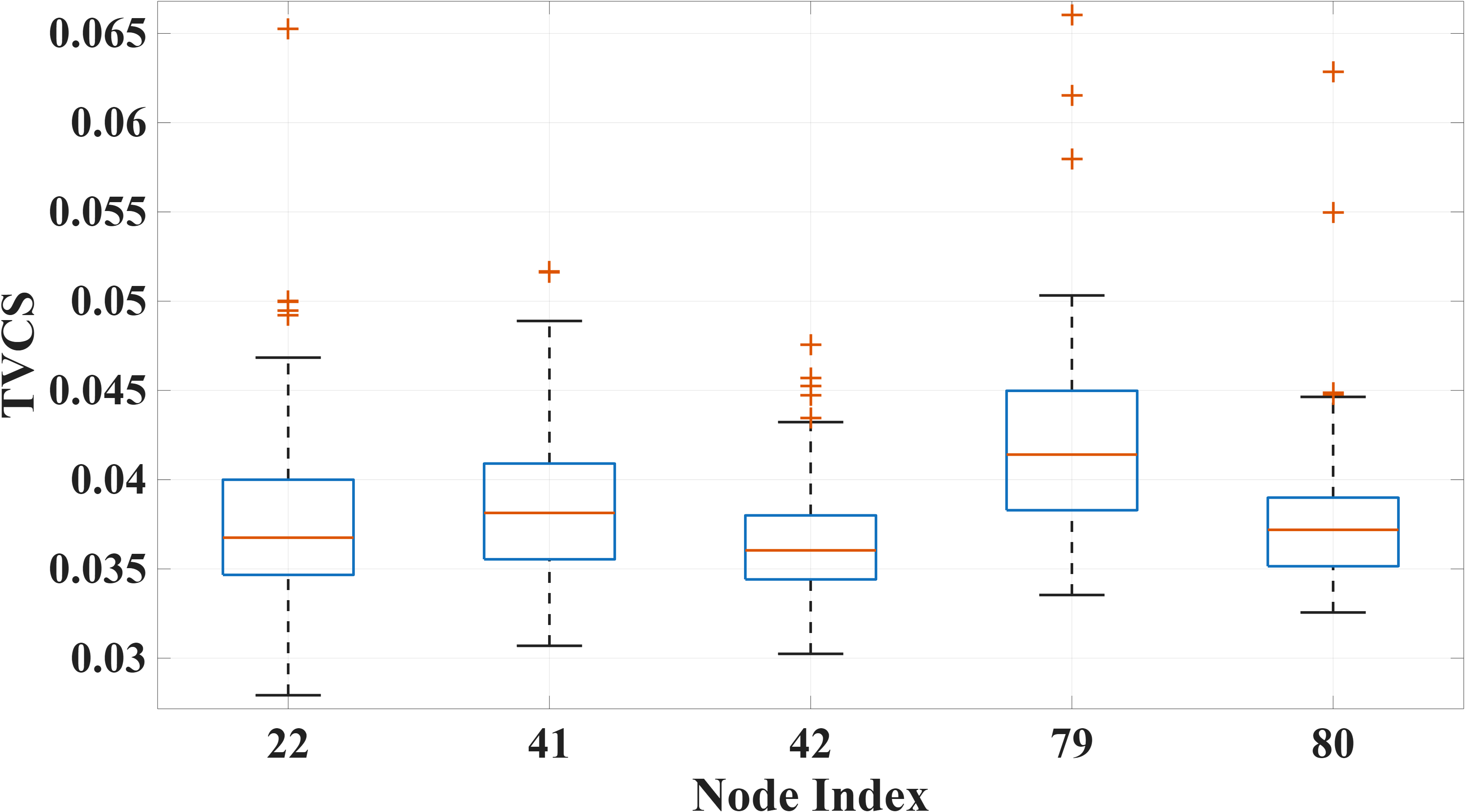}
   \end{minipage}
   \hfill
   \begin{minipage}[b]{0.49\linewidth}
      \centering
      \includegraphics[width=\linewidth]{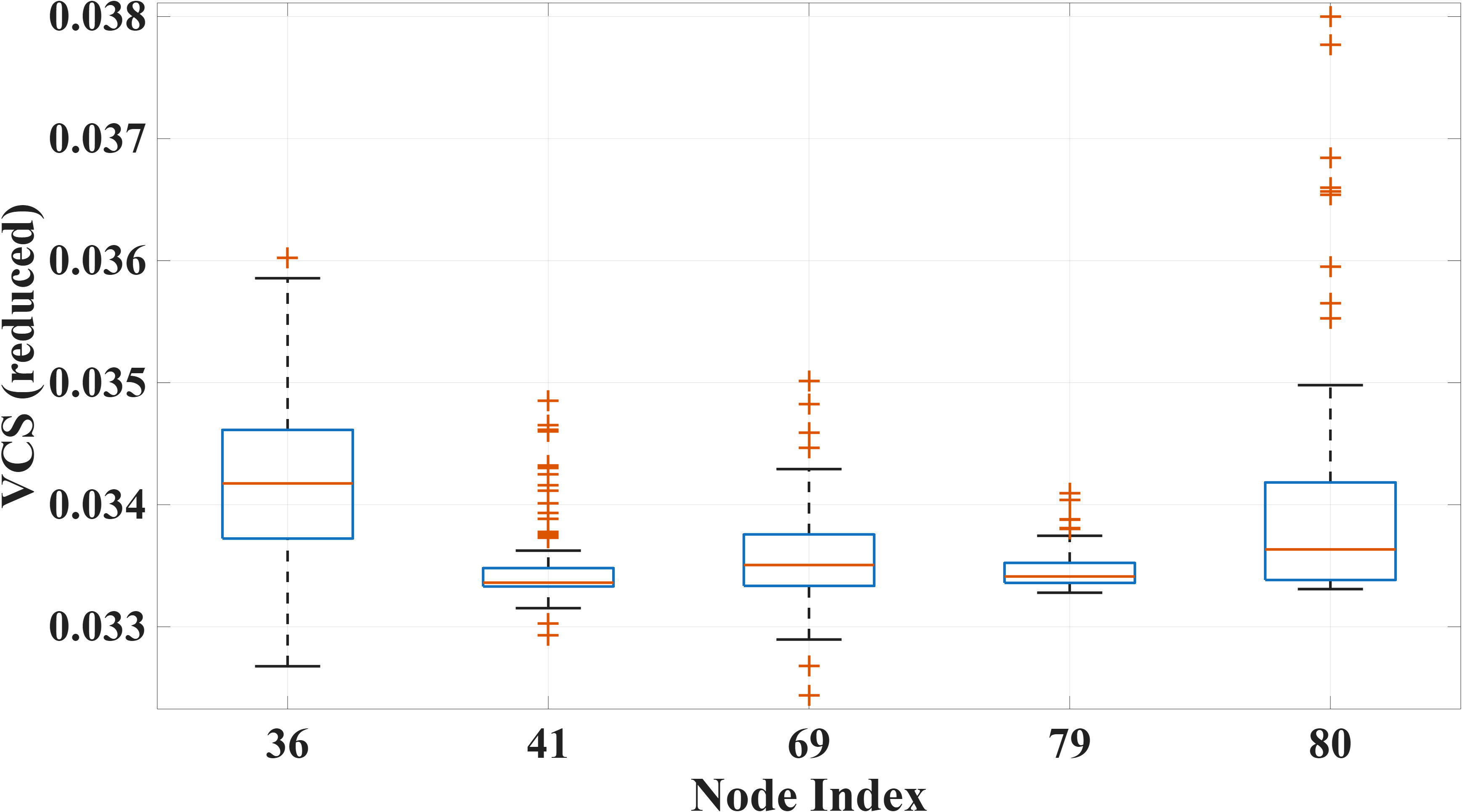}
   \end{minipage}
   \caption{Boxplots of the top 5 nodes for $(T,m)=(100,30)$: target VCS (left) and its reduced-system approximation (right).}
   \label{fig:VCS_T100_m30}
\end{figure}

\begin{table*}[t]
    \centering
    \caption{Node indices and brain regions in the top 5 rankings of TAECS or TVCS $(m=30)$.}
    \begin{tabular}{clccc}
        \toprule
        Node Index & Brain Region & TAECS $(T=1, 100)$ & TVCS $(T=1)$ & TVCS $(T=100)$ \\
        \midrule
         3  & Left Superior Frontal Gyrus   & $\checkmark$ &              &              \\
         4  & Right Superior Frontal Gyrus  & $\checkmark$ &              &              \\
        22  & Right Olfactory Cortex        &              &              & $\checkmark$ \\
        25  &  Left Middle Orbital Gyrus                            &              & $\checkmark$ &              \\
        36  & Right Cingulum Post     &              & $\checkmark$ &              \\
        41  & Left Amygdala                 &              &              & $\checkmark$ \\
        42  & Right Amygdala                &              &              & $\checkmark$ \\
        45  &   Left Cuneus       &              & $\checkmark$ &              \\
        68  & Right Precuneus               & $\checkmark$ &              &              \\
        73  & Left Putamen                  & $\checkmark$ &              &              \\
        74  & Right Putamen                 & $\checkmark$ &              &              \\
        75  &    Left Pallidum                        &              & $\checkmark$ &              \\
        76  &    Right Pallidum                        &              & $\checkmark$ &              \\
        79  & Left Heschl's Gyrus           &              &              & $\checkmark$ \\
        80  & Right Heschl's Gyrus          &              &              & $\checkmark$ \\
        \bottomrule
    \end{tabular}
    \label{tab:metrics_nodes_checked}
\end{table*}

%%%%%%%%%
\section{Conclusions} \label{Sec_conclusion}

\subsection{Summary} \label{Sec_summary}
This paper introduced the controllability allocation score (CAS)
as a general framework for allocating intervention intensity among
prescribed candidate input directions with respect to designated target
variables, and the target controllability score (TCS) as its nodewise
specialization.
We formulated volumetric and average-energy CASs as solutions to
finite-horizon optimization problems based on the output controllability
Gramian and developed a general uniqueness theory. In particular, for
the TCS specialization, we established uniqueness of the target VCS/AECS
for arbitrary system matrices and almost all positive horizons.
In contrast to the standard full-state setting, our analysis also showed
that target VCS/AECS can behave qualitatively differently, because
projection onto the target nodes changes the underlying Gramian structure.

For computation, we proposed a projected-gradient scheme for the
CAS framework and developed a general surrogate approach based on
approximate candidate output controllability Gramians. For the TCS
specialization, we further constructed a target-only reduced virtual
system.
Our error analysis shows how the approximation accuracy is
controlled by the coupling between target and non-target nodes and by
the dynamical growth rate, as quantified by the logarithmic norm.

Empirical studies on human brain networks (88 individuals) corroborate
the theoretical analysis. At short horizons, both target VCS and target
AECS are accurately approximated by their reduced counterparts. At
longer horizons, target AECS remains robust---consistently identifying
the same key regions---whereas target VCS becomes sensitive to time
scale and exhibits larger approximation discrepancies.

\subsection{Future Work Based on the Empirical and Theoretical Findings}

Building on the empirical and theoretical findings summarized in
Section~\ref{Sec_summary}, an important direction for future work is
to develop a horizon-robust theory for the target VCS/AECS, especially
in the large-\(T\) regime and for marginally stable or unstable systems.
In the standard full-state setting, recent work
\cite{umezu2026infinite} introduced a scaled controllability-Gramian
formulation that compensates for the growth of unstable modes, thereby
yielding optimization problems equivalent to the original finite-horizon
ones while remaining numerically stable for large \(T\).
Based on this reformulation, \cite{umezu2026infinite} established
well-defined infinite-horizon VCS/AECS, proved uniqueness under suitable
assumptions, and showed convergence of the finite-horizon scores as
\(T\to\infty\).
These results suggest that a similar large-\(T\) theory may also be
possible for the target setting.

However, such an extension is not immediate, because the target
VCS/AECS are defined not through the full-state controllability Gramian
itself but through an output controllability Gramian obtained after
projection onto the target nodes.
Owing to this projection, the resulting Gramian does not retain the
same direct relation to the spectral decomposition of \(A\) as in the
full-state case; rather, its behavior also depends on how the relevant
modes are reflected through the output map and on the coupling between
target and non-target nodes.
For this reason, the scaling argument and infinite-horizon
characterization in \cite{umezu2026infinite} do not carry over
automatically.

\subsection{Future Work on Intervention Analysis for Unknown Nonlinear Systems}
\label{future_nonlinear}

We explain why the target-oriented viewpoint developed in this paper
may also be useful as an intermediate step toward intervention analysis
for unknown nonlinear systems.
In many practical applications, the underlying dynamics are inherently
nonlinear, and one may consider a system of the form
\begin{align}
\dot{x}(t)=f(x(t)), \qquad x(t)\in\mathbb{R}^n.
\label{eq:nonlinear_system}
\end{align}
For such systems, the controllability-Gramian-based framework studied
in the present paper is not directly applicable.
Moreover, in data-scarce situations, it is often difficult to identify
a sufficiently accurate nonlinear model and to design detailed input
signals.
Nevertheless, one still needs to decide where and how strongly to
intervene.

One possible route is to introduce a lifted representation of
\eqref{eq:nonlinear_system}
based on Koopman operator theory for continuous-time nonlinear systems;
see, e.g., \cite{brunton2022modern}.
Suppose that, through a dictionary of observables
$\psi:\mathbb{R}^n \to \mathbb{R}^N$,
we define a lifted state
$z(t):=\psi(x(t))$.
Then the nonlinear system may admit, exactly or approximately, a
higher-dimensional linear surrogate model of the form
\begin{align}
\dot{z}(t)\approx A_{\psi}z(t),
\label{eq:koopman_surrogate}
\end{align}
where $A_{\psi}\in\mathbb{R}^{N\times N}$.
Such a surrogate may be identified directly from observed data,
without first identifying the nonlinear function \(f\).

One may then introduce virtual inputs into the lifted surrogate model
\eqref{eq:koopman_surrogate} for intervention-oriented analysis, and
consider the standard controllability scores proposed in
\cite{sato2022controllability} for the resulting system.
However, such an approach generally assigns importance not only to
components of $z(t)$ that correspond to the original state $x(t)$,
but also to auxiliary lifted coordinates that may have no direct
interpretation in the original nonlinear system
\eqref{eq:nonlinear_system}.
Therefore, the resulting score may not directly provide useful
guidance for intervention in the original state space.

To overcome this issue, it is natural to consider a dictionary that
explicitly includes the original state variables as part of the lifted
coordinates, for example,
$\psi(x)=
\begin{pmatrix}
x\\
\phi(x)
\end{pmatrix}
\in\mathbb{R}^N$,
where $\phi(x)\in\mathbb{R}^{N-n}$ denotes additional observables.
In this case, the first $n$ components of $z(t)$ coincide with the
original state variables $x(t)$.
Hence, instead of evaluating controllability for the entire lifted
state $z(t)$, one may focus only on those components that directly
correspond to the physically meaningful variables in
original system \eqref{eq:nonlinear_system}.

More specifically, by introducing a virtual input matrix $B(p)$ for
the lifted system
$\dot{z}(t)\approx A_{\psi}z(t)+B(p)u(t)$,
and an output equation
$y(t)=Cz(t)$,
where $C$ selects the components of $z(t)$ corresponding to the
original state variables of interest, one can define a
controllability allocation score focused on these target
components.
For example, if one is interested in all original state variables,
one may take
$C=
\begin{pmatrix}
I_n & 0
\end{pmatrix}$,
whereas if only a subset of the original variables is relevant to the
intervention objective, then $C$ can be chosen as the corresponding
selector matrix.
When the candidate virtual-input directions are further paired
with the selected original-state components, this construction reduces
to the TCS specialization developed in this paper.

In this way, the CAS framework can evaluate controllability
with respect to the target components that retain a clear
interpretation in original nonlinear system
\eqref{eq:nonlinear_system}, rather than all lifted
coordinates.
This may provide a more meaningful bridge from a lifted linear
surrogate model \eqref{eq:koopman_surrogate} to intervention analysis
for \eqref{eq:nonlinear_system}.
From this viewpoint, the target-oriented CAS, and the TCS in
the aligned specialization, can be interpreted as mechanisms for
extracting intervention-relevant information from a high-dimensional
lifted representation.

A remaining open question is how to choose the dictionary $\psi$ so
that the resulting target-oriented controllability allocation
score becomes both interpretable and useful for intervention-oriented
decision making.
Addressing this question is an important direction for future research.

\section*{Acknowledgment}
This work was supported by JST PRESTO, Japan, Grant Number JPMJPR25K4. 
The author used OpenAI's ChatGPT during the revision of this
manuscript for language editing and manuscript refinement.
All AI-assisted suggestions were independently reviewed and verified
by the author, who takes full responsibility for the content of the
manuscript.
%\appendix



\bibliographystyle{IEEEtran}
\bibliography{main.bib}

%\begin{thebibliography}{1}

  %0.5em minus 0.4em\relax Harlow, England: Addison-Wesley, 1999.

%\end{thebibliography}



%\vfill


\end{document}